\documentclass{amsart}

\usepackage{url,amssymb,enumerate,colonequals}

\usepackage{geometry}
\usepackage{mathrsfs} 
\usepackage[section]{placeins}
\usepackage{MnSymbol}
\usepackage{extarrows}
\usepackage{lscape}
\usepackage{circledsteps}
\usepackage[all,cmtip]{xy}

\usepackage[OT2,T1]{fontenc}
\usepackage{color}
\usepackage{tikzit}

\tikzstyle{white with black border}=[fill=white, draw=black, shape=circle]
\tikzstyle{black with white font}=[white, fill=black, draw=black, shape=circle, font={\small}, inner sep=0pt]
\tikzstyle{black simple style}=[fill=black, draw=black, shape=circle]

\tikzstyle{arrow}=[->]
\tikzstyle{new edge style 0}=[-, fill=none]

\usepackage{tikz-cd}
\usepackage{resizegather}
\usepackage[
        colorlinks=true, citecolor=darkgreen, linkcolor=darkgreen, urlcolor=darkgreen,
        backref=page,
        pdfauthor={Alvaro Gonzalez Hernandez},
]{hyperref}

\usepackage{comment}
\usepackage{multirow}

\usepackage{amssymb}
\usepackage{amsfonts}
\usepackage{amsmath}
\usepackage[nobysame,alphabetic]{amsrefs}

\usepackage{holtpolt}
\usepackage{xcolor}
\usepackage{adjustbox}
\usepackage{colortbl}
\usepackage{multicol}
\usepackage{appendix}
\usepackage{graphicx}
\usepackage{wrapfig}
\usepackage{stackrel}
\usepackage{changepage} 
\usepackage{listings}
\usepackage{xcolor}
\usepackage{float}
\usepackage{tikz-cd} 
\usepackage{mathtools}

\newcommand{\Z}{\mathbb{Z}}

\newcommand{\calA}{\mathcal{A}}

\newcommand{\calC}{\mathcal{C}}

\newcommand{\calH}{\mathcal{H}}

\newcommand{\calM}{\mathcal{M}}

\newcommand{\calO}{\mathcal{O}}

\newcommand{\CC}{\mathscr{C}}

\DeclareMathOperator{\Aut}{Aut}

\DeclareMathOperator{\Hom}{Hom}

\DeclareMathOperator{\Jac}{Jac}
\DeclareMathOperator{\Jacc}{\Jac(\mathcal{C})}

\DeclareMathOperator{\Spec}{Spec}

\newcommand{\Mtwo}{M_2}

\newcommand{\nor}{\calC_{E_1 \times_{\mathbb{P}^1} E_2}}

\newcommand{\GL}{\operatorname{GL}}

\newcommand{\PGL}{\operatorname{PGL}}

\newcommand{\SL}{\operatorname{SL}}

\newcommand{\floor}[1]{\left\lfloor #1 \right\rfloor}

\newcommand{\mell}{\mathcal{M}_{\text{ell}}}
\newcommand{\mellc}{\overline{\mathcal{M}}_{\text{ell}}}
\newcommand{\mells}{\mathcal{M}_{\text{ell}}/C_2}
\newcommand{\mquot}{\mathcal{M}/C_2}

\numberwithin{equation}{section}

\newtheorem{theorem}{Theorem}[section]  
\newtheorem{proposition}[theorem]{Proposition}

\newtheorem{corollary}[theorem]{Corollary}

\newtheorem{question}[theorem]{Question}

\theoremstyle{definition}

\newtheorem*{acknowledgements}{Acknowledgements}

\theoremstyle{remark}
\newtheorem{remark}[theorem]{Remark}

\definecolor{darkestblue}{HTML}{03045E}
\definecolor{darkgreen}{HTML}{006400}
\definecolor{darkblue}{HTML}{0077B6}
\definecolor{darkestred}{HTML}{c42706}
\definecolor{darkred}{HTML}{d44906}
\definecolor{lighterblue}{HTML}{00B4D8}

\DeclareRobustCommand{\SkipTocEntry}[5]{}

\makeatletter
\newcommand{\xdashrightarrow}[2][]{\ext@arrow 0359\rightarrowfill@@{#1}{#2}}
\newcommand{\xdashleftarrow}[2][]{\ext@arrow 3095\leftarrowfill@@{#1}{#2}}
\newcommand{\xdashleftrightarrow}[2][]{\ext@arrow 3359\leftrightarrowfill@@{#1}{#2}}
\def\rightarrowfill@@{\arrowfill@@\relax\relbar\rightarrow}
\def\leftarrowfill@@{\arrowfill@@\leftarrow\relbar\relax}
\def\leftrightarrowfill@@{\arrowfill@@\leftarrow\relbar\rightarrow}
\def\arrowfill@@#1#2#3#4{%
  $\m@th\thickmuskip0mu\medmuskip\thickmuskip\thinmuskip\thickmuskip
   \relax#4#1
   \xleaders\hbox{$#4#2$}\hfill
   #3$%
}
\makeatother
\begin{document}
\begin{abstract}
We compute the intersections between the automorphism strata and the pullback by the Torelli map of the Ekedahl-Oort strata inside the moduli space of genus two curves. We first describe explicitly which possible automorphism groups a genus two curve can have over a field of positive characteristic, and parametrise the families of curves with a prescribed automorphism group. Then, we describe an algorithm to compute the strata of genus two curves whose Jacobian variety has a fixed Ekedahl-Oort type. Finally, we compute the dimension and number of irreducible components of the intersections between the strata.
\end{abstract}

\title{Intersections of the automorphism and the Ekedahl-Oort strata in $M_2$}

\author{Alvaro Gonzalez-Hernandez}

\address{Mathematics Institute\\
    University of Warwick\\
    CV4 7AL \\
    United Kingdom\\}

\email{
\href{mailto:alvaro.gohe@outlook.com}{alvaro.gohe@outlook.com} 
}

\thanks{\emph{Website:} \url{https://alvarogohe.github.io}}
\maketitle

\section{Introduction}
The goal of this paper is to answer the following question: given a prime $p$, $0\leq f,a\leq 2$, and a group $G$, is there a genus two curve over a field of characteristic $p$ that has automorphism group $G$ and whose Jacobian variety has $p$-rank $f$ and $a$-number $a$? If the answer is affirmative, we would also like to obtain more information about the set of curves satisfying these conditions.\\

One possible approach is to study the locus of curves satisfying these properties inside the coarse moduli space of genus two curves $M_2$. The reason why this approach is convenient is because an explicit description of this moduli space is known, and we can perform computations in it using mathematical software, in our case Magma \cite{Bosma1997TheLanguage}. The code is available in \href{https://github.com/AlvaroGohe/Intersections-of-strata-in-the-moduli-space-of-genus-2-curves}{this repository}.\\

In Section \ref{M2_section}, we review the construction of an explicit model for the coarse moduli space of genus two curves using the Igusa invariants.\\

In Section \ref{automorphism_section}, we survey what the possible automorphism groups of a genus two curve are in every characteristic.  This section is of independent interest, as we provide a detailed description of the geometry of the automorphism strata and compute universal models for the curves in each stratum, extending the models described by Cardona and Quer \cite{Cardona2007CurvesD12} to work uniformly in every characteristic. In the case of curves with automorphism group $C_2^2$, we prove that this stratum can be related to a moduli space parametrising pairs of elliptic curves with level $2$ structures.\\

In Section \ref{EO_section}, we recall the description of the Ekedahl-Oort strata of the moduli space of abelian surfaces, and explain how to compute these strata from the description of the Hasse-Witt matrix.\\

Finally, in Section \ref{intersections_sec}, we present the main results of the chapter, which are tables containing the dimensions and the number of irreducible components of the strata corresponding to the intersections of the automorphism and the pullback of the Ekedahl-Oort strata inside $M_2$. \\

\begin{acknowledgements}
 I would like to thank Damiano Testa for his invaluable guidance and support. The idea for this project came after attending the course \textit{Geometry and arithmetic of moduli spaces of abelian varieties in positive characteristic} at the Arizona Winter School 2024. I would like to thank Valentijn Karemaker for teaching this course and for providing helpful comments on this paper. I was supported by the Warwick Mathematics Institute Centre for Doctoral Training and I gratefully acknowledge funding from the University of Warwick.
\end{acknowledgements}

\section{The coarse moduli space of genus two curves} \label{M2_section}
Let $M_2$ denote the coarse moduli space of genus two curves. Assuming that we are working over a field $k$ of characteristic not two, then, over the algebraic closure of $k$ every genus two curve can be written in the form
\begin{align} \label{genus2_eq}
y^2=\prod_{i=1}^6(x-\lambda_i),
\end{align}
for some $\lambda_i\in\overline{k}$ which are all distinct. In 1960, Igusa computed a set of five functions, now known as the \textbf{Igusa invariants} $\{J_2, J_4, J_6, J_8, J_{10}\}$, that determine a genus two curve up to isomorphism \cite{Igusa1960ArithmeticTwo}. These invariants are defined as symmetric functions depending on the $\lambda_i$ in the following way: \\

For any permutation $\sigma\in S_6$, we define $(ij)_\sigma:=(\lambda_{\sigma(i)}-\lambda_{\sigma(j)})$. Let $s_1,s_2,s_3$ and $s_4$ be the following symmetric polynomials of degrees $6$, $12$, $18$ and $30$:
\begin{align*}
s_1&= \sum_{\sigma\in S_6}(12)^2_\sigma(34)_\sigma^2(56)^2_\sigma,\\
s_2&=\sum_{\sigma\in S_6}(12)^2_\sigma(13)^2_\sigma(23)_\sigma^2(45)_\sigma^2(46)^2_\sigma(56)^2_\sigma,\\
s_3&=\sum_{\sigma\in S_6}(12)^2_\sigma(13)^2_\sigma(14)^2_\sigma(23)_\sigma^2(25)_\sigma^2(36)_\sigma^2(45)_\sigma^2(46)^2_\sigma(56)^2_\sigma,\\
s_4&=\prod_{1\leq i<j\leq6}(ij)_{\mathrm{id}}^2=\mathrm{disc}\Big(\prod_{i=1}^6(x-\lambda_i)\Big).\\
\end{align*}
Then, the Igusa invariants are defined by
\begin{align*}
J_2&=\tfrac{1}{2^3\, 3} s_1,\\
J_4&=\tfrac{1}{2^9\, 3^3} s_1^2-\tfrac{1}{3^3}s_2,\\
J_6&=\tfrac{1}{2^{13}\, 3^6} s_1^3+\tfrac{5}{2^{4}\, 3^6}s_1 s_2-\tfrac{2^4}{3^3}s_3,\\
J_8&=\tfrac{1}{2^{20}\, 3^7}s_1^4+\tfrac{13}{2^{10}\, 3^7}s_1^2 s_2-\tfrac{1}{2^1 3^4}s_1s_3-\tfrac{1}{2^{2}\, 3^6}s_2^2,\\
J_{10}&=2^8 s_4.\\
\end{align*}
Some sources define these invariants slightly differently, but we have chosen the definition that matches the implementation of the Igusa invariants in Magma \cite{Mestre1991ConstructionModules}.\\

As these only depend on symmetric polynomials of the $x$-coordinates of the Weierstrass points of the curve, one can easily check that the Igusa invariants can also be expressed in terms of the coefficients of the curve. Therefore, they are defined over the base field of the curve. The definition of these Igusa invariants extends to curves defined over any field of positive characteristic, in particular, over fields of characteristic two where genus two curves cannot be written as in the equation \ref{genus2_eq}.\\

These invariants are not algebraically independent, as for any curve we have the relation
\begin{align*}
J_4^2-J_2J_6+4J_8=0.\\
\end{align*}
Since our genus two curves are smooth, their discriminants do not vanish and, therefore, $J_{10}\neq 0$.
Let $\mathfrak{X}=\mathbb{V}(J_4^2-J_2J_6+4J_8)$. We can then define a morphism
\begin{align*}
M_2&\longrightarrow \mathfrak{X}\subset\mathbb{P}(1,2,3,4,5)\\
\mathcal{C}&\longmapsto[J_2(\calC):J_4(\calC):J_6(\calC):J_8(\calC):J_{10}(\calC)]
\end{align*}
and the image of $M_2$ inside $\mathbb{V}(J_4^2-J_2J_6+4J_8)$ corresponds to the open subvariety $\mathfrak{X}_0\subset\mathfrak{X}$ given by $J_{10}\neq0$. As stated before, the Igusa invariants determine a genus two curve up to isomorphism, and so, there is an inverse rational map from $\mathfrak{X}$ to $M_2$. Furthermore, the Deligne-Mumford compactification $\overline{M}_2$ is isomorphic to a blow-up of $\mathfrak{X}$ \cite[Théorème 2]{Liu1993CourbesModules}.\\

The method to recover a curve from its Igusa invariants has been described by Mestre \cite{Mestre1991ConstructionModules}. Note that, in general, given a set of Igusa invariants satisfying $J_4^2-J_2J_6+4J_8=0$ over a field $k$ that is not algebraically closed, it may not be possible to find a curve defined over $k$ with those invariants. \\

Let $\calA_2$ denote the coarse moduli space of principally polarised abelian surfaces. We define the \textbf{Torelli morphism} to be the map from $M_2$ to $\calA_2$ sending each curve to its Jacobian variety. This morphism is injective and its image is dense in $\mathcal{A}_2$ \cite{Milne1986JacobianVarieties}.\\

Since the moduli spaces $\mathcal{A}_g$ were first described, there has been significant interest in understanding their stratifications. These are decompositions of $\mathcal{A}_g$ into locally closed subsets known as \textbf{strata}, which correspond to abelian surfaces with specific geometric or arithmetic properties. For $g=2$, the Torelli morphism provides a means of translating these properties into the language of curves, allowing one to study the corresponding loci in $M_2$ as preimages of the relevant strata in $\mathcal{A}_2$.\\

\section{The automorphism group stratification} \label{automorphism_section}
From now on, we assume that we are working over an algebraically closed field $k$. Then, a genus two curve defined over a field of characteristic $p\geq0$ has one of the following automorphism groups \cite{Igusa1960ArithmeticTwo}:
\begin{itemize}
    \item \href{https://beta.lmfdb.org/Groups/Abstract/2.1}{$C_2$}, which happens generically.
    \item \href{https://beta.lmfdb.org/Groups/Abstract/4.2}{$C_2^2$}.
    \item \href{https://beta.lmfdb.org/Groups/Abstract/8.3}{$D_4$}, the dihedral group of order $8$, if $p\neq2$.
    \item \href{https://beta.lmfdb.org/Groups/Abstract/10.2}{$C_{10}$}, if $p\neq 2,5$.
    \item \href{https://beta.lmfdb.org/Groups/Abstract/12.4}{$D_{6}$}, the dihedral group of order $12$.
    \item \href{https://beta.lmfdb.org/Groups/Abstract/24.8}{$C_3\rtimes D_4$}, which has order $24$, if $p\neq 2,3,5$.
    \item \href{https://beta.lmfdb.org/Groups/Abstract/48.29}{$\GL_2(\mathbb{F}_3)$}, which has order $48$, if $p\neq 2,5$.\\
\end{itemize}
If $p=2$, we have the additional possibilities:
\begin{itemize}
    \item \href{https://beta.lmfdb.org/Groups/Abstract/32.51}{$C_2^5$}, which has order $32$.
    \item \href{https://beta.lmfdb.org/Groups/Abstract/160.235}{$C_2\wr C_5$}, which has order $160$.\\
\end{itemize}
If $p=5$, we also have:
\begin{itemize}
    \item \href{https://beta.lmfdb.org/Groups/Abstract/120.5}{$\SL_2(\mathbb{F}_5)$}, which has order $120$.\\
\end{itemize}

A consequence of the injectivity of the Torelli morphism is that every automorphism of $\Jacc$ preserving the polarisation arises from a unique automorphism of the curve $\calC$. Conversely, as every genus two curve is hyperelliptic, one can check that every automorphism of $\calC$ comes from a polarised automorphism of $\Jacc$ \cite[Theorem 12.1]{Milne1986JacobianVarieties}. Therefore, $\Aut(\calC)$ is isomorphic to the group of automorphisms of $\Jacc$ which preserve the polarisation. \\

Let $G$ denote one of the previous groups. Then, the subsets of curves whose automorphism group contains $G$ as a subgroup
\begin{align*}
W_{\geq G}=\{\calC\in \Mtwo(k): \Aut(\calC)\geq G\}
\end{align*}
are closed subschemes of $\Mtwo$. This can be seen from the fact that over algebraically closed fields, isomorphic curves have isomorphic automorphism groups and that the necessary conditions for a curve to admit a certain automorphism can be described in terms of the vanishing of polynomials involving the coefficients of the curve. Igusa proved that the subschemes $W_{\geq G}$ are equidimensional and their dimensions are the following:
\begin{table}[H]
\centering

\begin{tabular}{lccc|c|cc|c|cc|c|}
\cline{5-11}
                                                          &                                                    &                                                              &       & \cellcolor[HTML]{a7c957}$p\neq2$ & \multicolumn{2}{c|}{\cellcolor[HTML]{a7c957}$p\neq 2,5$}                      & \cellcolor[HTML]{a7c957}$p\neq 2,3,5$ & \multicolumn{2}{c|}{\cellcolor[HTML]{a7c957}$p=2$}                  & \cellcolor[HTML]{a7c957}$p=5$ \\ \hline
\rowcolor[HTML]{FFFFFF} 
\multicolumn{1}{|l|}{\cellcolor[HTML]{a7c957}$G$}         & \multicolumn{1}{c|}{\cellcolor[HTML]{ffffff}$C_2$} & \multicolumn{1}{c|}{\cellcolor[HTML]{ffffff}$C_2^2$} & $D_6$ & $D_4$                            & \multicolumn{1}{c|}{\cellcolor[HTML]{ffffff}$C_{10}$} & $\GL_2(\mathbb{F}_3)$ & $C_3\rtimes D_4$                      & \multicolumn{1}{c|}{\cellcolor[HTML]{ffffff}$C_2^5$} & $C_2\wr C_5$ & $\SL_2(\mathbb{F}_5)$         \\ \hline
\multicolumn{1}{|l|}{\cellcolor[HTML]{a7c957}$\dim(W_{\geq G})$} & \multicolumn{1}{c|}{3}                             & \multicolumn{1}{c|}{2}                                       & 1     & 1                                & \multicolumn{1}{c|}{0}                                & 0                     & 0                                     & \multicolumn{1}{c|}{1}                               & 0            & 0                             \\ \hline
\end{tabular}

\caption{Dimensions of the automorphism strata.}

\end{table}

\subsection{The zero-dimensional strata} \label{zerostrata_subsection}
Each of the zero-dimensional strata corresponds to a unique curve:
\begin{enumerate}
    \item Up to isomorphism, in characteristics not two or five, the unique curve with automorphism group $C_{10}$ is
    \begin{align*}
    y^2+y=x^5,
    \end{align*}
    which corresponds to the point $[0:0:0:0:1]$ in $\mathfrak{X}_0\subset\mathbb{P}(1,2,3,4,5)$. \\
    \item Up to isomorphism, in characteristics not two or five, the unique curve with automorphism $\GL_2(\mathbb{F}_3)$ is
    \begin{align*}
    y^2=x^5-x.
    \end{align*}
    In $\mathfrak{X}_0$, this curve corresponds to the point $[20: 30: -20: -325: 64]$.\\
    
    \item Up to isomorphism, in characteristics not two, three or five, the unique curve with automorphism group $C_3\rtimes D_4$ is
    \begin{align*}
    y^2=x^6+1.
    \end{align*}
     In $\mathfrak{X}_0$, this curve corresponds to the point $[120: 330: -320:-36825: 11664]$.\\
     
    \item Up to isomorphism, the unique curve in characteristic two with automorphism $C_2\wr C_5$ is
    \begin{align*}
    y^2+y=x^5.
    \end{align*}
 In $\mathfrak{X}_0$, this curve corresponds to the point $[0:0:0:0:1]$.\\
 
    \item Up to isomorphism, the unique curve in characteristic five with automorphism $\SL_2(\mathbb{F}_5)$ is
    \begin{align*}
    y^2=x^5-x.
    \end{align*}
    In $\mathfrak{X}_0$, this curve corresponds to the point $[0:0:0:0:1]$.\\
\end{enumerate}

\subsection{Describing the other automorphism strata}
The only strata that are left are the ones corresponding to the groups $C_2^2$, $D_4$ and $D_6$ in characteristic not two and $C_2^5$ in characteristic two.
The key idea to explicitly construct these strata is that we can characterise these automorphism groups from the presence of elements of a certain order. More specifically,
\begin{itemize}
    \item The genus two curves with automorphism group containing $C_2^2$ are those that admit a non-hyperelliptic automorphism of order two.
    \item  The genus two curves with automorphism group containing $D_4$ are those that admit an automorphism of order four.
    \item The genus two curves with automorphism group containing $D_6$ are those that admit an automorphism of order three.\\
\end{itemize}
In characteristic two, the curves in the stratum $W_{\geq C_2^5}$ can be easily described, as we will see in Subsection \ref{C2^5_subsection}.\\

Given any genus two curve, we can always transform it to one of the form
\begin{align*}
y^2+\sum_{i=0}^3g_ix^i y=\sum_{j=0}^6f_j x^j.
\end{align*}
for some $f_i,g_j\in k$. Genus two curves are always hyperelliptic, and therefore, always have a hyperelliptic involution $\iota$, which in this model it is given by
\begin{align*}
\iota\colon\quad y\mapsto -y-\sum_{i=0}^3g_ix^i.
\end{align*}

In order to construct the strata, we will assume that this general curve admits one of the special automorphisms stated above. In order for this to be possible, the coefficients $g_i$ and $f_j$ will have to satisfy certain relations. Then, using the description of the Igusa invariants in terms of the coefficients of the curve, we can obtain the equations for the locus of curves in $M_2$ satisfying these relations.\\

\subsection{The stratum of curves with automorphism group \texorpdfstring{$C_2^2$}{C2²}} \label{C22_section}
Suppose that a general genus two curve admits the automorphism
\begin{align} \label{order2inV4_eq}
\sigma\colon\quad\begin{cases}
    x\mapsto -x-1,\\
    y\mapsto y.
    \end{cases}
\end{align}

Then, we can always write such a curve in the form
\begin{align} \label{V4_eq}
\calC\colon\quad y^2+(g_1x(x+1)+g_0)y=f_3 x^3(x+1)^3+(f_2-f_1)x^2(x+1)^2+f_1x(x+1)+f_0
\end{align}
for some choice of values of $\{g_0,g_1,f_0,f_1,f_2,f_3\}$.\\

\begin{proposition} \label{strat_prop}
If the characteristic is two, $W_{\geq C_2^2}$ is given by $\mathbb{V}(h_{8},h_{9})\subset\mathfrak{X}_0$ where
\begin{align*}
h_8&=J_4^4 + J_4J_6^2 + J_2^2J_4J_8 + J_2^3J_{10},\\
h_9&=J_4^3J_6 + J_6^3 + J_2J_4^2J_8 + J_2^2J_4J_{10}.
\end{align*}
If the characteristic of the base field is not two, $W_{\geq C_2^2}$ is given by $\mathbb{V}(h_{15})\subset\mathfrak{X}_0$ where $h_{15}$ is a polynomial of weighted degree 15.\\
\end{proposition}
\begin{proof}
The same strategy employed here will later be used to derive explicit equations for the strata corresponding to Propositions \ref{j_prop}, \ref{D4_prop}, and \ref{D6_prop}. The Magma code used for all computations is available in the repository associated with this paper.\\

Using Magma, we first computed the expression for the Igusa invariants $J_i$ of the family of curves described in equation \ref{V4_eq} as polynomials in the variables $\{g_0,g_1,f_0,f_1,f_2,f_3\}$ with rational coefficients. We then determined the relations among these invariants, thus finding an expression for $W_{\geq C_2^2}$ as the vanishing locus of a polynomial $h_{15}\in\mathbb{Z}[J_2,J_4,J_6,J_8,J_{10}]$. This degree 15 polynomial, previously described by Shaska and Völklein \cite{Shaska2004EllipticFields}, characterizes the strata in characteristic zero. Since one can check that $\mathbb{V}(h_{15})$ is geometrically irreducible over $\mathbb{Q}$, and as the property of being geometrically irreducible is an open condition on the base scheme $\Spec(\mathbb{Z})$, it follows that for all but finitely many primes, the same polynomial $h_{15}$ over $\mathbb{F}_p$ defines the strata $W_{\geq C_2^2}$ in characteristic $p$. One can manually check the relations for the first hundred primes to find that the description for $W_{\geq C_2^2}$ only changes when the characteristic is two. 
\end{proof}\vspace{10pt}

\subsubsection{Parametrisation of the stratum} \label{par_subsection}
As mentioned in the proof of Proposition \ref{strat_prop}, regardless of the characteristic of the base field of definition, $W_{\geq C_2^2}$ is an irreducible surface inside $M_2$. Moreover, the singular subscheme of $W_{\geq C_2^2}$ corresponds to the union of all the strata $W_{\geq     H}$ with $H> C_2^2$, which, if the characteristic of the base field is two is $W_{\geq C_2^5}$, and if the characteristic of the base field is not two is $W_{\geq D_4}\cup W_{\geq D_6}$.\\

In characteristic two, we can see that $W_{\geq C_2^2}$ is rational, as we can compute a rational map
\begin{align*}
\mathbb{P}^2&\longrightarrow W_{\geq C_2^2}\subset\mathbb{P}(1,2,3,4,5)\\
[x_0:x_1:x_2]&\longmapsto[x_0:x_0 x_1:x_0 x_1^2: x_0^3x_2:x_1(x_0x_1^3 + x_1^4 + x_0^3x_2)],
\end{align*}
with inverse
\begin{align*}
W_{\geq C_2^2}&\longrightarrow \mathbb{P}^2\\
[J_2:J_4:J_6:J_8:J_{10}]&\longrightarrow [J_2^4:J_2^2J_4:J_8].\\
\end{align*}

In characteristic zero, it is also known that $W_{\geq C_2^2}$ is rational, for instance, by the work of Kumar \cite{Kumar2015HilbertFields}. Through similar methods, in every characteristic that is not two, we have computed a morphism
\begin{align*}
\mathbb{P}
(1,1,2)&\longrightarrow W_{\geq C_2^2}\subset\mathbb{P}(1,2,3,4,5)
\end{align*}
whose inverse is defined outside the singular locus of $W_{\geq C_2^2}$. This implies, in particular, that $W_{\geq C_2^2}$ is birationally equivalent to $\mathbb{P}
(1,1,2)$, and therefore rational.\\

Let $\sigma$ as in equation $\ref{order2inV4_eq}$ and $\calC\in W_{\geq C_2^2}$. Then, the quotients of $\calC$ by either the subgroup $\langle \sigma\rangle\subset\Aut(\calC)$ or  $\langle \iota \sigma\rangle\subset\Aut(\calC)$ are genus one curves, as both quotient maps ramify at two points. There is a natural choice that gives these quotients an elliptic curve structure, which is, for $\calC/\langle \sigma\rangle$, to set as the point at infinity the image of the fixed points by $\iota\sigma$, and for $\calC/\langle \iota\sigma\rangle$, to set as the point at infinity the image of the fixed points by $\sigma$. \\

An explicit model for $E_1=\calC/\langle \sigma\rangle$ is
\begin{align*}
E_1\colon\quad y^2+(g_1 x+f_3g_0)y=x^3+(f_2-f_1)x^2+f_1f_3x+f_0f_3^2,
\end{align*}
where the quotient morphism $\pi_1$ is given by 
\begin{align*}
\pi_1\colon\quad\calC&\longrightarrow E_1\\
(x,y)&\longmapsto (f_3x(x+1),f_3y)\\
\end{align*}
As for $E_2=\calC/\langle \iota\sigma\rangle$, if we let
\begin{align*}
\lambda=64f_0 - 20f_1 + 4f_2 - f_3 + 16g_0^2 - 8g_0g_1 + g_1^2,
\end{align*}
an explicit model of $E_2$ is given by
\begin{align*}
E_2\colon\quad y^2+((g_1-4g_0 )x-\lambda g_0)y=x^3+( 48f_0 - 9f_1 + f_2 + 8g_0^2 - 2g_0g_1)x^2+\lambda (12f_0 - f_1 + g_0^2)x+\lambda^2 f_0.\\
\end{align*}

Here,
\begin{align*}
\pi_2\colon\quad\calC&\longrightarrow E_2\\
(x,y)&\longmapsto \left(\frac{-\lambda x(x+1)}{(2x+1)^2}, \frac{\lambda(y+(x+1)(g_1 x(x+1)+g_0))}{(2x+1)^3}\right)\\
\end{align*}
Let $\gamma$ be a non-hyperelliptic involution in $\Jacc$, let $K_\calC$ denote the canonical divisor of $\calC$ and let
\begin{align*}
E_\gamma=\{D\in\Jacc:  D=(P)+(\gamma(P))-K_\calC\text{ with } P\in\calC\}.\\
\end{align*} 
Then, we have two isomorphisms $i_1:E_1\rightarrow E_\sigma $ and $i_2:E_2\rightarrow E_{\iota\sigma}$. Here is where we can see that our choice of the point at infinity in the $E_i$ is very natural, as the points $Q$ that are fixed by $\iota\gamma$ satisfy that $Q=\iota\gamma(Q)$ and so,
\begin{align*}
(Q)+(\gamma(Q))-K_\calC=(\iota\gamma(Q))+(\gamma(Q))-K_\calC=K_\calC-K_\calC=0.\\
\end{align*}

We therefore can construct an isogeny
\begin{align*}
\psi\colon\quad E_1\times E_2 &\longrightarrow \Jacc \\
(P_1,P_2)&\longmapsto i_1(P_1)+i_2(P_2)
\end{align*}
whose kernel is a subgroup of $\Jacc$ isomorphic to $(\mathbb{Z}/2\mathbb{Z})^2$ \cite{Artebani2014OnSurface}.\\

Let $j(E_1)$ and $j(E_2)$ be the $j$-invariants of $E_1$ and $E_2$ respectively. Then, using Magma one can check that for some $q_0, q_1, q_2\in k[\mathfrak{X}_0]$, $j(E_1)$ and $j(E_2)$ are both solutions of the quadratic equation
\begin{align*}
q_2j^2+q_1j+q_0=0.\\
\end{align*}

\begin{proposition} \label{j_prop}
In characteristic two, $j(E_1)=j(E_2)$ only when $J_2=0$, which corresponds to the stratum $W_{\geq C_2^5}$. In the rest of the characteristics, the subvariety of $W_{\geq C_2^2}$ for which $j(E_1)=j(E_2)$ forms two curves inside $\mathfrak{X}_0$, one of which corresponds to $W_{\geq D_4}$, and the other corresponding to the curve $Z=\mathbb{V}(h_5,h_6)$ where
\begin{align*}
h_5&=J_2^5 - 56J_2^3J_4 + 912J_2J_4^2 - 3456J_4J_6 + 576J_2J_8 + 17408J_{10},\\
h_6&=3J_2^4J_4 - 150J_2^2J_4^2 + 1871J_4^3 + 27J_6^2 + 73J_2^2J_8 - 1764J_4J_8 + 
        1904J_2J_{10}.
\end{align*}\newpage
If the characteristic of the base field is 17, $Z$ is given instead by $\mathbb{V}(h_3,h_{10})\subseteq\mathfrak{X}_0$ where
\begin{align*}
h_3 &= J_2^3 + 2J_2J_4 + 8J_6,\\
h_{10} &= 5J_4^5 + 10J_4^2J_6^2 + J_2^4J_4J_8 + 3J_4^3J_8 + 16J_6^2J_8 + 16J_4J_8^2 + 5J_2J_4^2J_{10} + 9J_4J_6J_{10} + 2J_2J_8J_{10} + 11J_{10}^2.\\
\end{align*}
\end{proposition}
\begin{proof}
The locus where $j(E_1)=j(E_2)$ corresponds to $\mathbb{V}(q_1)$ if the characteristic of the base field is two, and $\mathbb{V}(q_1^2-4q_1q_2)$ otherwise. One can check that in the second case, this is a reducible scheme that decomposes into the two curves $W_{\geq D_4}$. If the characteristic is zero, these two curves have models over $\mathbb{Z}$ and in all characteristics except when $p=17$, the equations for these curves are the reduction of the equations for these curves in characteristic zero.
\end{proof}

\subsubsection{Connection with the moduli space of pairs of elliptic curves with level 2 structures} \label{connection_subsection}
In this subsection, let us assume that the characteristic of the base field is not two.\\

We have seen that for each curve $\calC\in W_{\geq C_2^2}$, we can associate a pair of elliptic curves, namely, the quotients of $\calC$ by two non-hyperelliptic involutions $\sigma$ and $\iota\sigma$. A natural question is, is this pair of curves uniquely determined by the choice of $\calC$? The answer is no: if $\Aut(\mathcal{C}) = D_4$, then there are multiple choices of non-hyperelliptic involutions, and the resulting quotients can give rise to non-isomorphic pairs of elliptic curves, depending on the choice. An explicit example can be found in the accompanying Magma file in the repository.\\

We can also ask about the converse: given two elliptic curves, can we construct a genus two curve that is unique up to isomorphism whose quotients are those two curves? The answer to this question is also no, as for each pair of elliptic curves $\{E_1,E_2\}$ there are often multiple ways to glue those elliptic curves along their $2$-torsion that give rise to non-isomorphic genus two curves. More specifically, given a pair of elliptic curves, generically, there are six different non-isomorphic genus two curves whose quotients are those two elliptic curves. As we will see later on, these six curves correspond to the six possible ways that we can pair the non-trivial $2$-torsion points of $E_1$ and $E_2$.\\

However, despite all of this, if we fix extra information, we can still relate the stratum $W_{\geq C_2^2}$ with a coarse moduli space parametrising pairs of elliptic curves. We will now explain how.\\

Let $\calM$ be the coarse moduli space characterising isomorphic pairs $(\calC,\sigma)$ where 
\begin{itemize}
    \item The curve $\calC$ is in $W_{\geq C_2^2}$.
    \item $\sigma$ is a choice of a non-hyperelliptic involution of $\calC$.
    \item Two pairs $(\calC,\sigma)$ and $(\calC',\sigma')$ are isomorphic if there exists an isomorphism $\alpha:\calC\rightarrow\calC'$ such that $\sigma'=\alpha\circ\sigma\circ\alpha^{-1}.$\\
\end{itemize}

Let $X(2)$ be the coarse moduli space of elliptic curves with a \textbf{level 2 structure}.
The elements of this space correspond to pairs $(E,\phi)$ where 
\begin{itemize}
    \item $E$ is an elliptic curve.
    \item $\phi$ is an isomorphism $\phi:(\mathbb{Z}/2\mathbb{Z})^2\rightarrow E[2](k)$.
    \item Two pairs $(E,\phi)$ and $(E',\phi')$ are isomorphic if there exists an isomorphism $\alpha:E\rightarrow E'$ such that $\phi'=\alpha\circ\phi.$\\
\end{itemize}
 As a coarse moduli space $X(2)$ is isomorphic to $\mathbb{A}^1\setminus\{0,1\}$. This association comes from the fact that every elliptic curve $E$ can be written in Legendre form 
\begin{align*}
E_\lambda\colon\quad y^2=x(x-1)(x-\lambda)
\end{align*}
and we can identify $\lambda\leftrightarrow(E_\lambda,\phi_\lambda)$, where the morphism $\phi_\lambda$ is set to be the one that sends the two generators of $(\mathbb{Z}/2\mathbb{Z})^2$ to $P_0=(0,0)$ and $P_1=(1,0)$ in $E_\lambda[2]$.\newpage
Note that if we fix $E$, changing the structure $\phi$ corresponds to applying one of the following transformations $f_\tau\in \PGL_2(k)$ to $\lambda$:
\begin{align*}
f_{id}(\lambda)&=\lambda, & f_{(12)}(\lambda)&=1-\lambda, & f_{(13)}(\lambda)&=\frac{1}{\lambda},\\
f_{(23)}(\lambda)&=\frac{\lambda}{\lambda-1}, & f_{(123)}(\lambda)&=\frac{1}{1-\lambda}, & f_{(132)}(\lambda)&=\frac{\lambda-1}{\lambda}.\\
\end{align*}
Here, the notation we choose is such that for $\tau_1,\tau_2\in S_3$, $f_{\tau_1}\circ f_{\tau_2}=f_{\tau_1\circ\tau_2}$ and, therefore, this description allows us to see $S_3$ as a subgroup of $\PGL_2(k)$. Now, we define $\mellc$ to be $(X(2)\times X(2))/\Aut((\mathbb{Z}/2\mathbb{Z})^2)$, i.e. the quotient of $X(2)\times X(2)$ under the relation
\begin{align*}
((E_1,\phi_1),(E_2,\phi_2))\equiv ((E_1',\phi_1'),(E_2',\phi_2'))
\end{align*}
if and only if $E_1\cong E_1'$, $E_2\cong E_2'$, and there is an automorphism $\tau\in\Aut((\mathbb{Z}/2\mathbb{Z})^2)$ such that $\phi'_1=\phi_1\circ\tau$ and $\phi'_2=\phi_2\circ\tau$. Note that an automorphism of $(\mathbb{Z}/2\mathbb{Z})^2$ corresponds uniquely to a permutation of the three non-trivial $2$-torsion points, and as such, 
$\Aut((\mathbb{Z}/2\mathbb{Z})^2)\cong S_3$. We will fix this isomorphism to be the one corresponding to the ordering of the non-trivial points of $(\mathbb{Z}/2\mathbb{Z})^2$, $\{(1,0),(0,1),(1,1)\}$, so for the rest of the chapter we can identify any $\tau\in S_3$ as an element of $\Aut((\mathbb{Z}/2\mathbb{Z})^2)$ and vice versa.  \\

In terms of the description of $X(2)$ that we have given, we can describe
$\mellc$ as the quotient variety formed by the set of pairs $(\lambda_1,\lambda_2)\in(\mathbb{A}^1\setminus\{0,1\})^2$ under the relation $(\lambda_1,\lambda_2)\equiv(\lambda_1',\lambda_2')$ if there exists a $\tau\in S_3$ such that $(\lambda_1',\lambda_2')=(f_\tau(\lambda_1),f_\tau(\lambda_2))$.\\

Let $\Delta$ be the subvariety in $\mellc$ formed by the points of the form $((E,\phi),(E,\phi))$ for some $(E,\phi)\in X(2)$ and let $\mell=\mellc\setminus\Delta$. We then have the following result:

\begin{theorem} \label{iso_theorem}
There is an isomorphism $\Phi$ between $\calM$ and $\mell$.
\end{theorem}

\begin{proof}
For any $\calC\in W_{\geq C_2^2}$ and $\sigma$ a non-hyperelliptic involution, there are ten points that are fixed by some non-trivial element of the group $\langle\iota,\sigma\rangle$:
\begin{itemize}
    \item There are six points fixed by $\iota$, namely the Weierstrass points of the curve.  These points are not fixed by $\sigma$ or by $\iota\sigma$, but the action of $\sigma$ permutes them: applying $\sigma$ to a Weierstrass point yields a different Weierstrass point. Therefore, the six Weierstrass points decompose into three orbits under the action of $\sigma$, each consisting of two elements.
    \item There are two points fixed by $\sigma$, which are not fixed by $\iota$ or $\iota\sigma$.
    \item There are two points fixed by $\iota\sigma$, which are not fixed by $\iota$ or $\sigma$.
\end{itemize}

Recall that given a genus two curve, we can construct morphisms $\pi_1:\calC\rightarrow E_1$ and $\pi_2:\calC\rightarrow E_2$, where $E_1$ is the elliptic curve $\calC/\langle\sigma\rangle$
with the choice of the point at infinity given by the image by $\pi_1$ of the fixed points of $\iota\sigma$, and $E_2$ is the elliptic curve $\calC/\langle\iota\sigma\rangle$
with the choice of the point at infinity given by the image by $\pi_2$ of the fixed points of $\sigma$.\\

\underline{\textbf{The morphism $\Phi:\calM\rightarrow\mell$.}}\\

Let $(\calC,\sigma)\in\calM$. As we just mentioned, the action of $\iota$ on $\calC$ fixes the six Weierstrass points, and they form three orbits under the action of $\sigma$. Let us fix an ordering of these three orbits: 
\begin{align*}
\{\{P_{11},P_{12}\},\{P_{21},P_{22}\},\{P_{31},P_{32}\}\},
\end{align*}
where the $P_{ij}$ are the Weierstrass points.\\

Now, let $E_1=\calC/\langle\sigma\rangle$ and $E_2=\calC/\langle\iota\sigma\rangle$ as before. The action of $\iota$ on $\calC$ induces an action of order two on $E_i$, which  is the one that sends each point to its inverse with respect to the group law. The fixed points are the point of infinity and
\begin{align*}
Q_{i,1}&\coloneqq\pi_i(P_{11})=\pi_i(P_{12}), & Q_{i,2}&\coloneqq\pi_i(P_{21})=\pi_i(P_{22}), &
Q_{i,3}&\coloneqq\pi_i(P_{31})=\pi_i(P_{32}).
\end{align*}
Therefore, these points correspond to the non-trivial $2$-torsion points of $E_i$. For $i\in\{1,2\}$, we will then define $\phi_i$ to be the isomorphism:
\begin{align*}
\phi_i\colon\quad(\mathbb{Z}/2\mathbb{Z})^2&\longrightarrow E_i[2](k)\\
(1,0)&\longmapsto Q_{i,1}\\
(0,1)&\longmapsto Q_{i,2}\\
(1,1)&\longmapsto Q_{i,3}.\\
\end{align*}
Suppose that we picked another ordering of the orbits of the Weierstrass points. Let $\tilde{Q}_{i,j}$ the points on $E_i$ associated to this order and $\tilde{\phi_i}$ the corresponding level 2 structure. Then, there exists a $\tau\in S_3$ such that $Q_{i,\tau(j)}=\tilde{Q}_{i,j}$ and, as such, a $\tau\in \Aut((\mathbb{Z}/2\mathbb{Z})^2)$, such that $\tilde{\phi}_i=\phi_i\circ\tau$.\\

Furthermore, if we have two isomorphic pairs $(\calC,\sigma)$ and $(\calC',\sigma')$, there is an isomorphism $\alpha:\calC\rightarrow\calC'$ which induces isomorphisms $\calC/\langle\sigma\rangle\cong\calC'/\langle\alpha\sigma\alpha^{-1}\rangle=\calC'/\langle\sigma'\rangle$ and $\calC/\langle\iota\sigma\rangle\cong\calC'/\langle\alpha\iota\sigma\alpha^{-1}\rangle=\calC'/\langle\iota'\sigma'\rangle$. As $\alpha$ sends the fixed points of $\sigma, \iota\sigma$ and $\iota$ to the fixed points of $\sigma', \iota'\sigma'$ and $\iota'$ respectively, the associated elliptic curves satisfy that $(E_i,\phi_i)\cong(E'_i,\phi_i')$, for $i\in\{1,2\}$.\\

Therefore, the map
\begin{align*}
\Phi\colon\quad\calM&\longrightarrow\mellc\cong (X(2)\times X(2))/\Aut((\mathbb{Z}/2\mathbb{Z})^2)\\
(\calC,\sigma)&\longmapsto((E_1,\phi_1),(E_2,\phi_2))
\end{align*}
is well-defined. The image of $\Phi$ lies in $\mell$ for the following reason:\\

Given an elliptic curve $E_i$, we can always construct a degree two map that ramifies at the four $2$-torsion points. As we can map uniquely any two sets of three points in $\mathbb{P}^1$ using a projective transformation, we deduce that for each $E_i$ there exists a unique map $\psi_i: E_i\rightarrow\mathbb{P}^1$ satisfying
\begin{align*}
\psi_i(\phi_i((1,0)))&=0, & \psi_i(\phi_i((0,1)))&=1, &  \psi_i(\phi_i((1,1)))&=\infty.
\end{align*}
Composing $\psi_i$ with the quotient map $\pi_i:\calC\rightarrow E_i$ gives us maps $\psi_i\circ\pi_i:\calC\rightarrow\mathbb{P}^1$ which satisfy that $\psi_1\circ\pi_1=\psi_2\circ\pi_2$.\\

If $(E_1,\phi_1)\cong(E_2,\phi_2)$, this would imply that there is a $\lambda\in\mathbb{A}^1\setminus\{0,1\}$ and two isomorphisms $\alpha_i:E_i\rightarrow E_\lambda$ such that $\phi_\lambda=\alpha_i\circ\phi_i$. 
One can check that if we define $\psi'$ to be the map $E_\lambda\rightarrow\mathbb{P}^1$ such that
\begin{align*}
\psi'(\phi_\lambda((1,0)))&=0, & \psi'(\phi_\lambda((0,1)))&=1, &  \psi'(\phi_\lambda((1,1)))&=\infty,
\end{align*}
then $\psi'=\psi_1\circ\alpha_1^{-1}=\psi_2\circ\alpha_2^{-1}$, and so,
\begin{align*}
\psi'\circ\alpha_1\circ\pi_1=\psi_1\circ\alpha_1^{-1}\circ\alpha_1\circ\pi_1=\psi_1\circ\pi_1=\psi_2\circ\pi_2=\psi_2\circ\alpha_2^{-1}\circ\alpha_2\circ\pi_2=\psi'\circ\alpha_2\circ\pi_2.\\
\end{align*}
Let $Q$ and $\iota(Q)$ be the two fixed points of $\iota\sigma$ in $\calC$. Then, $\alpha_1\circ\pi_1$ would send both points to the point at infinity of $E_\lambda$, which is a ramification point of $\psi'$. From the description of the fixed points that we gave at the start of the proof, one can check that $\alpha_2\circ\pi_2$ would send $Q$ and $\iota(Q)$ to two different points of $E_\lambda$, none of which is a ramification point of $\psi'$. As $\psi'$ has degree two, we deduce that 
\begin{align*}
\psi'\circ\alpha_1\circ\pi_1(Q)&\neq\psi'\circ\alpha_2\circ\pi_2(Q) & \psi'\circ\alpha_1\circ\pi_1(\iota(Q))&\neq\psi'\circ\alpha_2\circ\pi_2(\iota(Q)) 
\end{align*}
and this would lead to a contradiction, as we had shown that $\psi'\circ\alpha_1\circ\pi_1=\psi'\circ\alpha_2\circ\pi_2$.\\[10pt]

\underline{\textbf{The morphism $\Psi:\mell\rightarrow\calM$.}}\\

Let $((E_1,\phi_1),(E_2,\phi_2))\in X(2)\times X(2)$ such that $(E_1,\phi_1)\not\cong(E_2,\phi_2)$. To construct the inverse of $\Phi$, consider the two morphisms $\psi_1$ and $\psi_2$ that we defined before, which, as we saw, depended on the $\phi_i$.  We define the \textbf{fibre product} of $(E_1,\phi_1)$ and $(E_2,\phi_2)$ along $\mathbb{P}^1$ to be the scheme $E_1\times_{\mathbb{P}^1}E_2$, together with morphisms $p_1$ and $p_2$ to $E_1$ and $E_2$ respectively that satisfies the following universal property: for any scheme $W$
with morphisms $\varphi_1$ and $\varphi_2$, such that $\psi_1\circ\varphi_1=\psi_2\circ\varphi_2$, there exists a unique morphism $\theta:W\rightarrow E_1\times_{\mathbb{P}^1}E_2$ such that $\varphi_1=p_1\circ\theta$ and $\varphi_2=p_2\circ\theta$.

\[\begin{tikzcd}
W \arrow[rd, "\exists\theta" description] \arrow[rrd, "\varphi_1"] \arrow[rdd, "\varphi_2"'] &                                                                &                          \\
                                                                                      & E_1\times_{\mathbb{P}^1}E_2 \arrow[r, "p_1"'] \arrow[d, "p_2"] & E_1 \arrow[d, "\psi_1"'] \\
                                                                                      & E_2 \arrow[r, "\psi_2"]                                        & \mathbb{P}^1            
\end{tikzcd}\vspace{10pt}\]

We will prove that the normalisation $\nor$ of $E_1\times_{\mathbb{P}^1}E_2$ is a genus two curve. To do this, we will show that for any $E_1$ and $E_2$, there is a genus two curve $\calC_{\lambda_1,\lambda_2}$ and morphisms $\varphi_1$ and $\varphi_2$ making the diagram commutative and use the information that we know about the ramification of the morphisms to show that this curve is isomorphic to $\nor$.\\

Let $\mu_i\in\mathbb{P}^1$ be the image under $\psi_i$ of the point at infinity of $E_i$, and let $\lambda_i=\tfrac{1}{1-\mu_i}$. As we saw, we can construct an isomorphism $\alpha_i:E_i\rightarrow E_{\lambda_i}$ where
\begin{align*}
 E_{\lambda_i}\colon\quad y^2=x(x-1)(x-\lambda_i)
\end{align*}
such that $(E_i,\phi_i)\cong(E_{\lambda_i},\phi_{\lambda_i})$ and $\psi_i=\psi_i'\circ\alpha_i$ is defined by
\begin{align*}
\psi_i'\colon\quad E_{\lambda_i}&\longrightarrow\mathbb{P}^1\\
(x,y)&\longmapsto\frac{(1-\lambda_i)x}{x-\lambda_i}\\
\end{align*}
Recall that $(E_1,\phi_1)\not\cong(E_2,\phi_2)$ implies that $\lambda_1\neq\lambda_2$. Now, consider the curve
\begin{align*}
\calC_{\lambda_1,\lambda_2}\colon\quad y^2=(x^2-1)\left(x^2-\frac{\lambda_1}{\lambda_2}\right)\left(x^2-\frac{\lambda_1(\lambda_2-1)}{\lambda_2(\lambda_1-1)}\right).\\
\end{align*}
The curve $\calC_{\lambda_1,\lambda_2}$ has genus two whenever its discriminant is non-zero. This occurs precisely for any two values $\lambda_1,\lambda_2\in \mathbb{A}^1\setminus\{0,1\}$, such that $\lambda_1\neq\lambda_2$. We can construct two maps, $\rho_1$ and $\rho_2$
\begin{align*}
\rho_1\colon\quad \calC_{\lambda_1,\lambda_2} &\longrightarrow E_{\lambda_1} \\
(x,y) &\longmapsto\left(\nu_1^2\left(x^2-\frac{\lambda_1(\lambda_2-1)}{\lambda_2(\lambda_1-1)}\right),\nu_1^3 y\right),\\
\rho_2\colon\quad \calC_{\lambda_1,\lambda_2} &\longrightarrow E_{\lambda_2} \\
(x,y) &\longmapsto\left(-\nu_2^2\left(\frac{1}{x^2}-\frac{\lambda_2(\lambda_1-1)}{\lambda_1(\lambda_2-1)}\right),\nu_2^3\left( \frac{\lambda_2(\lambda_1-1)^{\frac{1}{2}}y}{\lambda_1(\lambda_2-1)^{\frac{1}{2}}x^3}\right)\right),
\end{align*}

where
\begin{align*}
\nu_1&=\left(\frac{\lambda_2(\lambda_1-1)}{\lambda_1-\lambda_2}\right)^\frac{1}{2}, & \nu_2&=\left(\frac{\lambda_1(\lambda_2-1)}{\lambda_1-\lambda_2}\right)^\frac{1}{2}.\\
\end{align*}
Then, $\psi_1'\circ\rho_1=\psi_2'\circ\rho_2$, as in both cases, this is the map
\begin{align*}
\psi_i'\circ\rho_i\colon\quad \calC_{\lambda_1,\lambda_2}&\longrightarrow \mathbb{P}^1\\
(x,y)&\longmapsto \frac{\lambda_2(1-\lambda_1)x^2-\lambda_1(1-\lambda_2)}{\lambda_2x^2-\lambda_1}.\\
\end{align*}
As a consequence, we can define  $\varphi_1=\alpha_1^{-1}\circ\rho_1$ and $\varphi_2=\alpha_2^{-1}\circ\rho_2$, and deduce that there exists a map $\theta:\calC_{\lambda_1,\lambda_2}\rightarrow E_1\times_{\mathbb{P}^1}E_2$ making the following diagram commutative:
\[
\begin{tikzcd}
\calC_{\lambda_1,\lambda_2} \arrow[rd, "\theta" description] \arrow[rrd, "\varphi_1"] \arrow[rdd, "\varphi_2"'] \arrow[rr, "\rho_1"] \arrow[dd, "\rho_2"'] &                                                                & E_{\lambda_1} \arrow[dd, "\psi_1'" description, bend left] \\
                                                                                                                                        & E_1\times_{\mathbb{P}^1}E_2 \arrow[r, "p_1"'] \arrow[d, "p_2"] & E_1 \arrow[d, "\psi_1"'] \arrow[u, "\alpha_1"]               \\
E_{\lambda_2} \arrow[rr, "\psi_2'" description, bend right]                                                                             & E_2 \arrow[r, "\psi_2"] \arrow[l, "\alpha_2"']                   & \mathbb{P}^1                                              
\end{tikzcd}\vspace{10pt}
\]
It is not difficult to see that the preimages of the three points that are simultaneously in the branch loci of $\psi_1$ and $\psi_2$ give rise to singularities of $E_1\times_{\mathbb{P}^1}E_2$, so this is not a smooth curve. As $p_1$ is a dominant morphism to $E_1$ of degree two, $E_1\times_{\mathbb{P}^1}E_2$ has at most two irreducible components.\\

If it had two irreducible components, the ramification locus of $p_1$ should correspond to the points where the two components meet and, therefore, all of these points should be singular. However, we can check that $p_1$ has two ramification points which are smooth, namely, the two preimages of the ramification point of $\psi_1$ that is not a ramification point of $\psi_2$ (earlier in the proof we denoted this point as $\mu_1$). Therefore, we deduce that $E_1\times_{\mathbb{P}^1}E_2$ is irreducible. Although we will not prove it, one can further check that $E_1\times_{\mathbb{P}^1}E_2$ has three nodal singularities, and its arithmetic genus is five.\\

As both maps $\varphi_1$ and $p_1$ have degree two, we deduce that the map $\theta$ is generically $1$-to-$1$, and as $\calC_{\lambda_1,\lambda_2}$ is smooth, we deduce that $\calC_{\lambda_1,\lambda_2}$ is isomorphic to the normalisation of $E_1\times_{\mathbb{P}^1}E_2$. There is a clear non-hyperelliptic involution
in $\calC_{\lambda_1,\lambda_2}$ sending $x\mapsto -x$, and this gives rise to an involution $\tilde{\sigma}$ in $\nor$.
Therefore, we have a morphism
\begin{align*}
(X(2)\times X(2))\setminus\Delta&\longrightarrow\calM\\
((E_1,\phi_1),(E_2,\phi_2))&\longmapsto(\nor,\tilde{\sigma})\\
\end{align*}
This extends to a morphism $\Psi:\mell\rightarrow\calM$.\\

To see this, assume that we have $((E'_1, \phi'_1), (E'_2, \phi'_2))$ such that $E_1 \cong E'_1$, $E_2 \cong E'_2$, and there is an automorphism $\tau \in \Aut((\mathbb{Z}/2\mathbb{Z})^2)$ with $\tau(\phi_i) = \phi'_i$ for $i = 1, 2$.  
Then, from the description of $X(2)$ that we gave earlier, we see that $\mathcal{C}_{E'_1 \times_{\mathbb{P}^1} E'_2} \cong \mathcal{C}_{f_\tau(\lambda_1), f_\tau(\lambda_2)}$.  
By computing the Igusa invariants with Magma, we can check that 
\[
J_{2n}(\mathcal{C}_{\lambda_1, \lambda_2}) = J_{2n}(\mathcal{C}_{f_\tau(\lambda_1), f_\tau(\lambda_2)})
\]
for all $1 \leq n \leq 5$ and all $\tau \in S_3$. Therefore, $\mathcal{C}_{\lambda_1, \lambda_2} \cong \mathcal{C}_{f_\tau(\lambda_1), f_\tau(\lambda_2)}$.\\

Hence, $\Psi: \mathcal{M}_{\mathrm{ell}} \to \mathcal{M}$ is a well-defined morphism.

\begin{remark}
While it is always true that $\mathcal{C}_{\lambda_1, \lambda_2} \cong \mathcal{C}_{f_\tau(\lambda_1), f_\tau(\lambda_2)}$, it is certainly not true in general that $\mathcal{C}_{\lambda_1, \lambda_2} \cong \mathcal{C}_{\lambda_1, f_\tau(\lambda_2)}$.  
In fact, one can check that $\mathcal{C} \in W_{\geq D_4}$ if and only if there exists $\lambda \in \mathbb{A}^1 \setminus \{0,1\}$ and an involution $\tau \in S_3$ such that $\mathcal{C} \cong \mathcal{C}_{\lambda, f_\tau(\lambda)}$.  
Likewise, $\mathcal{C} \in Z$, where $Z$ is as defined in subsection~\ref{par_subsection}, if and only if there exists $\lambda$ and an element $\tau \in S_3$ of order three such that $\mathcal{C} \cong \mathcal{C}_{\lambda, f_\tau(\lambda)}$.\\
\end{remark}
\newpage

\underline{$\Phi=\Psi^{-1}$}\\

Finally, to prove that we have an isomorphism $\Phi:\calM\rightarrow\mell$, we need to check that $\Psi\circ\Phi=\mathrm{id}_{\calM}$. This is easy to see from the following commutative diagram:
\[\begin{tikzcd}
\calC \arrow[rrd, "\pi_1"] \arrow[rdd, "\pi_2"'] &                                                                                                                   &                                                 \\
                                                 & \quad\calC/\langle\sigma\rangle\!\times_{\mathbb{P}^1}\!\calC/\langle\iota\sigma\rangle \arrow[r, "p_1"'] \arrow[d, "p_2"] & \calC/\langle\sigma\rangle \arrow[d, "\psi_1"'] \\
                                                 &\calC/\langle\iota\sigma\rangle \arrow[r, "\psi_2"]                                                               & \mathbb{P}^1                                   
\end{tikzcd}\]

The universal property of the fibre product implies that we can construct a map $\theta:\calC\rightarrow\calC/\langle\sigma\rangle\times_{\mathbb{P}^1}\calC/\langle\iota\sigma\rangle$. This map $\theta$ induces an isomorphism between $\calC$ and the normalisation of $\calC/\langle\sigma\rangle\times_{\mathbb{P}^1}\calC/\langle\iota\sigma\rangle$. By carefully tracking the involution $\sigma$ through the commutative diagrams, one can check that $\Psi\circ\Phi=\mathrm{id}_{\calM}$.
\end{proof}

For every $\calC\in W_{\geq C_2^2}$ and every non-hyperelliptic involution $\sigma$, we have that both $(\calC,\sigma)$ and $(\calC,\iota\sigma)$ are in $\calM$, so there is an action of $C_2$ on $\calM$ exchanging these two points of the moduli space. Let $\mquot$ denote the quotient. Similarly, we have an action of $C_2$ on $\mell$ that swaps the two copies of $X(2)$. Let $\mells$ be the quotient of $\mell$ by that action. We can easily check that 
\begin{align*}
\Phi((C,\sigma))&=((E_1,\phi_1),(E_2,\phi_2))& &\Longleftrightarrow& \Phi((C,\iota\sigma))&=((E_2,\phi_2),(E_1,\phi_1))
\end{align*}
and therefore $C_2$ acts on $\calM$ and $\mell$ in a compatible way. As a consequence, we deduce the following:
\begin{corollary}
As coarse moduli spaces, $\mquot$ is isomorphic to $\mells$.
\end{corollary}
We can define a forgetful morphism $\pi:\mquot\rightarrow W_{\geq C_2^2}$ that sends the equivalence class of a pair $(\calC,\sigma)$ to $\calC$. By checking the conjugacy class of every element of order two in every possible automorphism group of $\calC$, we can study the fibres of this morphism $\pi$:
\begin{enumerate}
    \item If $\calC\in W_{\geq C_2^2}\setminus W_{\geq D_4} $, we can check that there are only two classes of pairs of points in $\calM$ whose first factor is $\calC$, corresponding to $(\calC,\sigma)$ and $(\calC,\iota\sigma)$ for some choice of $\sigma$. These choices are identified when we take the quotient, and therefore $\pi^{-1}(\calC)$ is only one point.
    \item If $\calC\in W_{\geq D_4}\setminus W_{\geq \GL_2(\mathbb{F}_3)}$, there are again two classes of pairs of points in $\calM$ whose first factor is $\calC$. If we let $\sigma$ be a non-hyperelliptic involution and $\tau_4$ an automorphism of order four of $\calC$, the two pairs are $(\calC,\sigma)$ and $(\calC,\tau_4\,\sigma)$. In both cases, $(\calC,\sigma)\cong(\calC,\iota\sigma)$ and $(\calC,\tau_4\,\sigma)\cong(\calC,\iota\tau_4\,\sigma)$, so they are different points in $\mquot$. Therefore, $\pi^{-1}(\calC)$ are two points.
    \item If $\Aut(\calC)=\GL_2(\mathbb{F}_3)$, up to conjugation, $\calC$ has a unique non-hyperelliptic involution, and therefore $\pi^{-1}(\calC)$ is only one point. 
    
\end{enumerate}
From this analysis, we deduce that the forgetful morphism $\pi: \mquot\rightarrow W_{\geq C_2^2}$ is injective outside the set of points where the image of this map lies inside $W_{\geq D_4}\setminus W_{\geq \GL_2(\mathbb{F}_3)}$. We will use this result in Section \ref{intersections_sec}.\\

\subsection{The stratum of curves with automorphism group \texorpdfstring{$D_4$}{D4}}
Assume that the characteristic of the base field $k$ is not two. Without any loss of generality, suppose that a general genus two curve admits the automorphism of order four
\begin{align*}
\tau_4\colon\quad\begin{cases}
    x\mapsto -x,\\
    y\mapsto iy.
    \end{cases}
\end{align*}

Then, we can always write such a curve in the form
\begin{align*}
y^2=f_5x^5+f_3x^3+f_1x 
\end{align*}
for some value of $\{f_1,f_3,f_5\}$.\\

Using the same strategy as in Proposition \ref{strat_prop}, one can prove the following:

\begin{proposition} \label{D4_prop}
If the characteristic of $k$ is not five, $W_{\geq D_4}$ is given by $\mathbb{V}(h_5,h_6)$ where
\begin{align*}
h_5&=J_2^5 - 64J_2^3J_4 + 1216J_2J_4^2 - 5760J_4J_6 + 768J_2J_8 - 128000J_{10}\\
h_6&=J_2^4J_4 - 74J_2^2J_4^2 + 1456J_4^3 - 4320J_6^2 - 40J_2^2J_8 + 1728J_4J_8 - 3200J_2J_{10}.
\end{align*}
If the characteristic of $k$ is five, then $W_{\geq D_4}$ is given by $\mathbb{V}(h_4,h_8,h_9)$ with
\begin{align*}
    h_4&=J_2^4 + J_2^2J_4 + J_4^2 + 3J_8,\\
    h_8&=2J_2^2J_4^3 + 2J_4^4 + 3J_4J_6^2 + 3J_2^2J_4J_8 + J_4^2J_8 + 2J_8^2 + J_2^3J_{10},\\
    h_9&=J_6^3 + 2J_2J_4^2J_8 + 3J_2J_8^2 + J_2^2J_4J_{10}.\\
\end{align*}
\end{proposition}

The curve $W_{\geq D_4}$ always has a singular point with coordinates $[120: 330: -320:-36825: 11664]$. From this, we can see that it is in $\mathfrak{X}\setminus \mathfrak{X}_0$ if the characteristic is three, it corresponds to the curve with automorphism group $\SL_2(\mathbb{F}_5)$ if the characteristic is five, and, otherwise, it corresponds to the curve with automorphism group $C_3\rtimes D_4$.\\

\subsubsection{Parametrisation of the stratum}
Regardless of the characteristic of the base field, the stratum $W_{\geq D_4}$ is a rational curve, as we can find a birational morphism $\mathbb{P}^1\rightarrow W_{\geq D_4}$ whose inverse is given by the map
\begin{align*}
W_{\geq D_4} &\longrightarrow \mathbb{P}^1\\
[J_2:J_4:J_6:J_8:J_{10}]&\longmapsto[J_2(11J_2^2 - 480J_4): 5J_2^3 - 224J_2J_4 - 720J_6]
\end{align*}
which is defined outside of the singular point of $W_{\geq D_4}$.\\

From the previous map that we described, we see that if $P=[J_2:J_4:J_6:J_8:J_{10}]$ is a smooth point of $W_{\geq D_4}$, a model for the curve corresponding to this point is
\begin{align*}
y^2=x^5+x^3+\left(\frac{7 J_2^3 - 288 J_2 J_4 + 2160 J_6}{4 (3 J_2^3 - 160 J_2 J_4 - 3600 J_6)}\right)x.\\
\end{align*}
This expression generalises the one found by Cardona and Quer, in the sense that it is valid in every characteristic that is not two \cite[Proposition 2.1]{Cardona2007CurvesD12}.\\

\subsection{The stratum of curves with automorphism group \texorpdfstring{$D_6$}{D6}} \label{D6_subsection}
Suppose a general genus two curve admits the automorphism of order three
\begin{align*}
\tau_3\colon\quad\begin{cases}
    x\mapsto \frac{1}{1-x},\\
    y\mapsto \frac{y}{(x-1)^3}.
    \end{cases}\\
\end{align*}
Then, we can always write such a curve in the form $y^2+g(x)y=f(x)$ with
\begin{align*} 
g(x)&=g_0 x^3-(g_1+3g_0)x^2+g_1 x+g_0,\\
f(x)&=f_0 x^6-(f_1+6 f_0 )x^5+(f_2+5f_1+15f_0)x^4-(2f_2+5f_1+10f_0)x^3+f_2 x^2+f_1x +f_0,
\end{align*}
for some choice of values of $\{g_0,g_1,f_0,f_1,f_2\}$.\\

\begin{proposition} \label{D6_prop}
If the characteristic is three, $W_{\geq D_6}$ is given by $\mathbb{V}(h_{2},h_{5})\subset\mathfrak{X}_0$ where
\begin{align*}
h_2&=J_4,\\
h_5&=J_2J_8-J_{10},
\end{align*}
and if the characteristic is five, $W_{\geq D_6}$ is given by $\mathbb{V}(h_{4},h_{8}, h_9)\subset\mathfrak{X}_0$ where
\begin{align*}
h_4&= J_2J_6 + 4J_8,\\
h_8&=3J_4J_6^2 + 4J_2^2J_4J_8 + J_8^2 + J_2^3J_{10},\\
h_9&=J_6^3 + J_4J_6J_8 + 3J_2J_8^2 + J_2^2J_4J_{10}.
\end{align*}
Otherwise, $W_{\geq D_6}$ is given by $\mathbb{V}(h_{5},h_{6})\subset\mathfrak{X}_0$ where
\begin{align*}
h_5&=2J_2^2J_6 - 45J_4J_6 - 2J_2J_8 - 375J_{10},\\
h_6&=-175J_4^3 + 9J_2^3J_6 - 675J_6^2 - 9J_2^2J_8 - 540J_4J_8.\\
\end{align*}
\end{proposition}
This curve has a singular point with coordinates $[20: 30: -20: -325: 64]$ in every characteristic except three. This point is in $\mathfrak{X}\setminus\mathfrak{X}_0$ if the characteristic is two, corresponds to the curve with automorphism group $\SL_2(\mathbb{F}_5)$ if the characteristic is five, and, otherwise, corresponds to the curve with automorphism group $\GL_2(\mathbb{F}_3)$.\\

\subsubsection{Parametrisation of the stratum}
As before, regardless of the characteristic of the base field, the stratum $W_{\geq D_6}$ is a rational curve, as we can find a birational morphism $\mathbb{P}^1\rightarrow W_{\geq D_6}$.\\

The inverse is given by 
\begin{align*}
W_{\geq D_6} &\longrightarrow \mathbb{P}^1\\
[J_2:J_4:J_6:J_8:J_{10}]&\longmapsto[J_2^3: J_6]
\end{align*}
in characteristic three, and 
\begin{align*}
W_{\geq D_6} &\longrightarrow \mathbb{P}^1\\
[J_2:J_4:J_6:J_8:J_{10}]&\longmapsto[J_2(3J_2^2 - 40J_4): -J_2J_4 - 30J_6]
\end{align*}
otherwise. This map is except at the singular point of $W_{\geq D_6}$.\\

Similarly to the $W_{\geq D_4}$ stratum, all the points $W_{\geq D_6}$ over $k$ correspond to genus two curves defined over $k$. If the characteristic is not three, we can see that if $P$ is a smooth point of $W_{\geq D_6}$ with coordinates $[J_2:J_4:J_6:J_8:J_{10}]$, a model for the curve corresponding to this point is
\begin{align*}
y^2+(x^3+1)y=-x^6-\frac{3 J_2^3 - 133 J_2 J_4 - 2790 J_6}{3 (3 J_2^3 - 160 J_2 J_4 - 3600 J_6)}.\\
\end{align*}
This expression generalises the one found by Cardona and Quer, as it is valid in every characteristic that is not three \cite[Proposition 2.1]{Cardona2007CurvesD12}.\\

Extending our base field to include a third root of unity $\omega$, we can find a model for this curve as in the start of Subsection \ref{D6_subsection} where the values for $\{f_0,f_1,f_2,g_0,g_1\}$ are given by
\begin{align*}
f_0&=\frac{12 J_2^3 - 613 J_2 J_4 - 13590 J_6}{81(3 J_2^3 - 160 J_2 J_4 - 3600 J_6)},\\
f_1&=\frac{2 \omega ((3 + 9 \omega)J_2^3-(133 + 480 \omega)J_2 J_4-(2790+10800\omega)J_6)}{27(3 J_2^3 - 160 J_2 J_4 - 3600 J_6)},\\
f_2&=\frac{5\omega(J_2^3-(480+133\omega)J_2J_4+(10800+2790\omega)J_6)}{27(3 J_2^3 - 160 J_2 J_4 - 3600 J_6)},\\
g_0&=0,\\
g_1&=1.\\
\end{align*}
This model of the curve does not reduce well modulo three, but the Igusa invariants define a point in $M_2$ that has good reduction at $p=3$. The reduction of that point modulo three corresponds to the following genus two curve:
\begin{align*}
y^2+(-x^2+x)y=-\frac{J_6^\frac{1}{3}}{J_2}(x+1)^6.\\
\end{align*}

\subsection{The stratum of curves with automorphism group \texorpdfstring{$C_2^5$}{C2⁵}} \label{C2^5_subsection}
Suppose that the automorphism group of a genus two curve in characteristic two contains $C_2^5$. Then, by the work of Igusa \cite{Igusa1960ArithmeticTwo}, we can see that this curve must be of the form
\begin{align*}
y^2+g_0y=f_5x^5+f_3 x^3
\end{align*}
for some value of $\{g_0,f_3,f_5\}$.\\

The stratum $W_{\geq C_2^5}$ is given by $\mathbb{V}(J_2,J_4,J_6)\subset\mathfrak{X}_0$, so it is parametrised to be the points of the form $[0:0:0:J_8:J_{10}]\in M_2$. The curve corresponding to the point $[0:0:0:J_8:J_{10}]$ corresponds to setting the values of $\{g_0,f_3,f_5\}$ to
\begin{align*}
g_0&=J_{10}^{\tfrac{1}{8}}, & f_3&=J_{8}^{\tfrac{1}{8}}, &f_5&=1,
\end{align*}
in the previous model.\\

\section{The Ekedahl-Oort stratification} \label{EO_section}

Let $\Jacc$ be the Jacobian variety of a genus two curve $\calC$ defined over an algebraically closed field $k$ of characteristic $p$. The \textbf{$\boldsymbol{p}$-rank} of $\Jacc$ is the integer $f$ such that $\lvert \Jacc[p](k)\rvert=p^f$. As $\dim(\Jacc)=2$, $0\leq f\leq 2$.\\

Let $\mu_p$ denote the finite group scheme representing $R\rightarrow \Spec(R[x]/(x^p-1))$ for any ring $R$. We can also define the $p$-rank in terms of the group scheme $\mu_p$ as
\begin{align*}
f(\Jacc)=\dim_{\mathbb{F}_p}(\Hom(\mu_p,\Jacc)).\\
\end{align*}

Let $\alpha_p$ denote the finite group scheme representing $R\rightarrow \Spec(R[x]/(x^p))$ for any ring $R$ of characteristic $p$. Then, the \textbf{$\boldsymbol{a}$-number} of $\Jacc$ is defined to be
\begin{align*}
a(\Jacc)=\dim_k(\Hom(\alpha_p,\Jacc)).\\
\end{align*}
Note that $0\leq a\leq 2$ and $0\leq a+f\leq2$.\\

In positive characteristic, it is well-known that the moduli space of principally polarised varieties can always be stratified by the $p$-rank of the abelian variety. This stratification can be refined in many ways, such as intersecting these strata with the ones corresponding to abelian varieties with a fixed $a$-number, or by considering the stratification in terms of the Newton polygon of the variety.\\

However, the information provided by all those strata is determined by an even finer way of stratifying the moduli space in terms of an invariant known as the \textbf{Ekedahl-Oort type}.\\

The idea behind this invariant is that the $p$-torsion of an abelian variety is a symmetric truncated Barsotti-Tate group scheme of level one ($BT_1$), and the Ekedahl-Oort type classifies its isomorphism class as a $BT_1$ group scheme. This information is encoded in combinatorial information corresponding to the possible filtrations of the Dieudonné module associated to the $p$-torsion.\\

The theory behind these strata is incredibly rich, and gives rise to many fine strata for abelian varieties of high dimension \cites{Ekedahl2009CycleVarieties, Oort1974SubvarietiesSpaces}. However, in the case of abelian surfaces, we will not need to get into the details, as there are only four Ekedahl-Oort strata, and they all correspond to the intersection of strata with $p$-rank $f$ and $a$-number $a$, as shown by Pries \cite{Pries2008AP}:

\begin{table}[h]
\centering
\begin{tabular}{|c|c|c|c|}
\hline
\rowcolor[HTML]{a7c957} 
$\text{Type}$                    & $f$ & $a$ & $\dim$ \\ \hline
$(\Z/p\Z)^2\oplus\mu_p^2$                   & $2$ & $0$ & $3$                 \\ \hline
$\Z/p\Z\oplus\mu_p\oplus I_{1,1}$       & $1$ & $1$ & $2$                 \\ \hline
$I_{2,1}$               & $0$ & $1$ & $1$                 \\ \hline
$(I_{1,1})^2$ & $0$ & $2$ & $0$                 \\ \hline
\end{tabular}\vspace{10pt}

\caption{Types of Ekedahl-Oort strata for abelian surfaces.}
\end{table}

Here, $I_{r,1}$ is defined to be the unique symmetric $BT_1$ group scheme of rank $p^{2r}$ with $p$-rank 0 and $a$-number 1.\\

\subsection{The Hasse-Witt matrix of a genus two curve}
We will now see how to compute the $p$-rank and $a$-number of the Jacobian of a genus two curve $\calC$. Let $\sigma:k\rightarrow k$ be the Frobenius automorphism, let $\tau$ be its inverse, and let $H^0(\calC,\Omega^1_\calC)$ be the space of regular $1$-forms on $\calC$, which is a two-dimensional $k$-vector space that admits the basis $B=\{\frac{dx}{y},\frac{x dx}{y}\}$. The \textbf{Cartier operator} is the semi-linear map $\CC:H^0(\calC,\Omega^1_\calC)\rightarrow H^0(\calC,\Omega^1_\calC)$ satisfying the following properties:
\begin{enumerate}
    \item $\CC(\omega_1+\omega_2)=\CC(\omega_1)+\CC(\omega_2)$.
    \item $\CC(f^p\omega)=f\CC(\omega)$.
    \item $\CC(f^{n-1} df)=\begin{cases}
    df &\text{if }n=p,\\
    0 &\text{otherwise.}
    \end{cases}$\\
\end{enumerate}

We can check that the Cartier operator is a $\tau$-linear operator, whose associated matrix in the basis $B$ is what we know as the \textbf{Cartier-Manin} matrix $M_{CM}$. Here, we work under the convention that left multiplication by $M_{CM}$ calculates the effect of $\CC$ in the basis $B$ \cite{Achter2019HasseWittRequest}.\\

On the other hand, by Serre duality, the choice of $B$ fixes a basis $B'$ for the dual space $H^1(\calC,\calO_\calC)$. The action of Frobenius on $H^1(\calC,\calO_\calC)$ gives then rise to a $\sigma$-linear operator, whose associated matrix in the basis $B'$ is known as the \textbf{Hasse-Witt matrix} $M_{HW}$.
The way these two matrices are related is
\begin{align*}
M_{HW}&=(M_{CM}^\sigma)^\intercal, & M_{CM}&=(M_{HW}^\tau)^\intercal.\\
\end{align*}
In characteristic $p\geq3$, there is an easy way to compute these two matrices. If we complete the square to make $\calC$ to be of the form $y^2=f(x)$, for some $f(x)$ of degree either five or six, and if we define the $c_i$ to be the constants satisfying the relation
\begin{align*}
\sum_{i=0}c_i x^i=f(x)^{\frac{p-1}{2}},
\end{align*}
we can check that the Hasse-Witt matrix $M_{HW}$ is given by
\begin{align*}
M_{HW}=\begin{pmatrix}
c_{p-1} & c_{2p-1}\\
c_{p-2} & c_{2p-2}
\end{pmatrix}.\\
\end{align*}

The $p$-rank and $a$-number of the Jacobian of a genus two curve can be easily computed from the Hasse-Witt matrix by the following result:
\begin{proposition} \label{HW_prop} \cite[Lemma 1.1]{Ibukiyama1986SupersingularNumbers}
The Ekedahl-Oort type of $\Jacc$ is completely determined by its $p$-rank $f$ and its $a$-number $a$, and hence by its Hasse-Witt matrix:
\begin{itemize}
    \item $(f,a)=(2,0)$ if and only if $M_{HW}$ has rank two.
    \item $(f,a)=(1,1)$ if and only if $M_{HW}$ has rank one and $M_{HW}M_{HW}^\sigma\neq 0$.
    \item $(f,a)=(0,1)$ if and only if $M_{HW}\neq 0$ and $M_{HW}M_{HW}^\sigma= 0$.
    \item $(f,a)=(0,2)$ if and only if $M_{HW}= 0$.\\
\end{itemize}
\end{proposition}
For any $0 \leq f,a \leq 2$, the sets
\begin{align*}
\{X \in \calA_2(k) &: f(X) \leq f\}, & \{X \in \calA_2(k) &: a(X) \geq a\},
\end{align*}
are closed subschemes of $\calA_2$ of pure dimension $f+1$ and $3-\floor{\tfrac{a(a+1)}{2}}$ respectively \cite[Theorem 14.7]{vanderGeer1999CyclesVarieties}. As the Hasse-Witt matrix of $\Jacc$ can be determined from the coefficients of $\calC$, we can also deduce that both
\begin{align*}
V_{\leq f}&=\{\calC \in M_2(k) : f(\Jacc) \leq f\}, &
T_{\geq a}&=\{\calC \in M_2(k) : a(\Jacc) \geq a\},
\end{align*}
form closed subschemes of $M_2$, as both the $p$-rank and the $a$-number can be expressed as the vanishing of polynomials involving the coefficients of $M_{HW}$. From the fact that the Torelli map from $M_2$ to $\calA_2$ is injective and dense, we deduce that $\dim(V_{\leq f})=f+1$ and $\dim(T_{\geq a})=3-\floor{\tfrac{a(a+1)}{2}}$.\\

\subsection{Computing the Ekedahl-Oort strata}
Using Proposition \ref{HW_prop}, we propose an algorithm to compute the equations for the pullback of the Ekedahl-Oort strata and their intersections with the automorphism strata inside $M_2$.\\

Assuming that the characteristic of the base field is not two, given any genus two curve, we can always find a projective transformation that sends three of the Weierstrass points to $(0,0),(1,0)$ and $\infty$ and, as such, we can always transform any curve into one of the form
\begin{align*}
\calC\colon\quad y^2=x(x-1)(a_3x^3+a_2 x^2+a_1x+a_0).\\
\end{align*} 
Note that if we replace all the $a_i$ by $\lambda a_i$ for some $\lambda\in k^\times$, we obtain a curve that is also isomorphic to $\calC$. Therefore, we can consider $\calC$ as the point $[a_0:a_1:a_2:a_3]$ inside $\mathbb{P}^3$. The discriminant of $\calC$ is a degree ten polynomial in the variables $a_i$, and outside of the vanishing locus of this polynomial, each point of $\mathbb{P}^3$ corresponds to a point in $M_2$. Therefore, we can define a rational map
\begin{align*}
\phi\colon\quad\mathbb{P}^3\longrightarrow\mathfrak{X}_0
\end{align*}
by sending a point $[a_0:a_1:a_2:a_3]$ to the Igusa invariants corresponding to the curve defined by those coefficients. Let $M_{HW}$ denote the Hasse–Witt matrix associated to $\calC$, which has entries in $k[a_0,a_1,a_2,a_3]$. Let $(f,a)$ be the $p$-rank and $a$-rank corresponding to an Ekedahl-Oort type. We define the following subvarieties of $\mathbb{P}^3$:
\begin{align*}
X_{(1,1)} &= \mathbb{V}(\det(M_{HW})),\\
X_{(0,1)} &= \mathbb{V}(M_{HW} M_{HW}^\sigma),\\
X_{(0,2)} &= \mathbb{V}(M_{HW}).
\end{align*}

By Proposition \ref{HW_prop},  it is easy to check that $\phi(X_{(f,a)})=V_{\leq f}\cap T_{\geq a}$.\\

This is quite an inefficient way to compute which curves lie in the supersingular strata, and there are better algorithms in the literature for this purpose, such as the one developed by Pieper \cite{Pieper2022ConstructingJacobian}. However, this algorithm can be used quite successfully to compute the stratum $V_{\leq 1}=V_{\leq 1}\cap T_{\geq 1}$.\\

Assuming we are not working in characteristic two, $\mathfrak{X}\cong\mathbb{P}(1,2,3,5)$. Then, using this algorithm, we can show computationally that $V_{\leq 1}$ can be described inside $\mathbb{P}(1,2,3,5)$ as a reduced and irreducible hypersurface of degree $\frac{1}{2}(p-1)$, at least for all primes $p\leq71$.\\

In the Magma file, we have also managed to compute the singularities of these surfaces and the intersections with the stratum $W_{\geq C_2^2}$. These computations have been very useful for gathering evidence for the statements that are later proved in the following section.\\

Another consequence of these computations is the following result:
\begin{proposition}
The stratum $V_{\leq 1}$ is rational if $p\leq19$.
\end{proposition}
\begin{proof}
Given a well-formed quasismooth hypersurface $X_d$ of weighted degree $d$ inside a weighted projective space whose highest weight is $w$, Esser proved that $X_d$ is rational whenever $d<2w$ \cite[Proposition 3.1]{Esser2024RationalHypersurfaces}. In our case, given that $d=\frac{1}{2}(p-1)$ and $w=5$, we obtain that $V_{\leq 1}$ is rational when $p<21$.
\end{proof}
The fact that the degree of $V_{\leq 1}$ grows linearly with $p$ is a sign that there may be a prime $p_0$ such that $V_{\leq 1}$ is not rational whenever $p\geq p_0$. However, $V_{\leq 1}$
is closely related to the determinantal variety of a $2\times 2$ matrix, which is a rational surface, so it is also conceivable that $V_{\leq 1}$ remains rational for all primes.\\

Therefore, we would be interested to know:
\begin{question}
\textit{Is $V_{\leq 1}$ rational in characteristic $p$ for all prime numbers?\\}
\end{question}

\section{The intersections of the strata} \label{intersections_sec}
In the previous sections we have described the automorphism and the Ekedahl-Oort strata. Now we will proceed to describe the intersections of these. Consider
\begin{align*}
W_{= G}&=\{\calC\in \Mtwo(k): \Aut(\calC)= G\},\\ V_{=(f,a)}&=\{\calC\in \Mtwo(k): f(\Jacc)=f \text{ and } a(\Jacc)=a\}.\\
\end{align*}
We can see that these are all strata of $M_2$, i.e. locally closed schemes, because we can express them as
\begin{align*}
W_{= G}&=W_{\geq G}\setminus\bigcup_{H> G} W_{\geq H},\\
V_{=(f,a)}&=\left(V_{\leq f}\cap T_{\geq a}\right)\setminus \left(\left(V_{\leq f-1}\cap T_{\geq a}\right)\cup\left(V_{\leq f}\cap T_{\geq a+1}\right)\right),
\end{align*}
where $V_{\leq-1}=T_{\geq3}=\emptyset$. Furthermore, $\overline{W_{= G}}=W_{\geq G}$.\\

\subsection{Dimensions of the intersections}
We will now identify the empty strata $ V_{=(f,a)} \cap W_{= G}$, and compute the dimension of those which are non-empty. In all of the cases, the strata are equidimensional. 

\begin{theorem} \label{intersection_theorem}
Let $G$ be the automorphism group of a genus two curve over an algebraically closed field of characteristic $p$, and let $f, a$ be two integers between zero and two. Then, the dimensions of the strata $V_{=(f,a)}\cap W_{= G}$, i.e. the loci in $M_2$ of  genus two curves with $p$-rank $f$, $a$-number $a$ and automorphism group $G$, are collected in the following tables for each possible value of $p$.\\

When $p=2$,
\begin{table}[H]
\centering
\begin{tabular}{cc|c|c|c|c|c|}
\cline{3-7}
                                                & &\cellcolor[HTML]{a7c957}{\color[HTML]{333333} $C_2$} & \cellcolor[HTML]{a7c957}{\color[HTML]{333333} $C_2^2$} & \cellcolor[HTML]{a7c957}{\color[HTML]{333333} $D_6$} & \cellcolor[HTML]{a7c957}{\color[HTML]{333333} $C_2^5$} & \cellcolor[HTML]{a7c957}{\color[HTML]{333333} $C_2\wr C_5$} \\ \hline
\multicolumn{1}{|c|}{\cellcolor[HTML]{a7c957}$2$} & \cellcolor[HTML]{a7c957}$0$ &{\color[HTML]{333333} $3$}                             & {\color[HTML]{333333} $2$}                                       & {\color[HTML]{333333} $1$}                             & {\color[HTML]{C0C0C0}  $-1$}                              & {\color[HTML]{C0C0C0}  $-1$}                                   \\ \hline
\multicolumn{1}{|c|}{\cellcolor[HTML]{a7c957}$1$} & \cellcolor[HTML]{a7c957}$1$ &{\color[HTML]{333333} $2$}                             & {\color[HTML]{C0C0C0}  $-1$}                                      & {\color[HTML]{C0C0C0}  $-1$}                            & {\color[HTML]{C0C0C0}  $-1$}                              & {\color[HTML]{C0C0C0}  $-1$}                                   \\ \hline
\multicolumn{1}{|c|}{\cellcolor[HTML]{a7c957}$0$} & \cellcolor[HTML]{a7c957}$1$ &{\color[HTML]{C0C0C0}  $-1$}                            & {\color[HTML]{C0C0C0}  $-1$}                                      & {\color[HTML]{C0C0C0}  $-1$}                            & {\color[HTML]{333333} $1$}                               & {\color[HTML]{333333} $0$}                                    \\ \hline
\multicolumn{1}{|c|}{\cellcolor[HTML]{a7c957}$0$} & \cellcolor[HTML]{a7c957}$2$ &{\color[HTML]{C0C0C0}  $-1$}                            & {\color[HTML]{C0C0C0}  $-1$}                                      & {\color[HTML]{C0C0C0}  $-1$}                            & {\color[HTML]{C0C0C0}  $-1$}                                & {\color[HTML]{C0C0C0}  $-1$}                                    \\ \hline
\end{tabular}
\caption{Dimensions of $V_{=(f,a)}\cap W_{= G}$ when $p=2$.}
\end{table}

When $p=3$,
\begin{table}[H]
\centering
\begin{tabular}{cc|c|c|c|c|c|c|}
\cline{3-8}
                                                &                           & \cellcolor[HTML]{a7c957}$C_2$ & \cellcolor[HTML]{a7c957}$C_2^2$ & \cellcolor[HTML]{a7c957}$D_4$ & \cellcolor[HTML]{a7c957}$D_6$ & \cellcolor[HTML]{a7c957}$\GL_2(\mathbb{F}_3)$ & \cellcolor[HTML]{a7c957}$C_{10}$ \\ \hline
\multicolumn{1}{|c|}{\cellcolor[HTML]{a7c957}$2$} & \cellcolor[HTML]{a7c957}$0$ & $3$                             & $2$                                       & $1$                             & $1$                             & $0$                    & {\color[HTML]{C0C0C0} $-1$}        \\ \hline
\multicolumn{1}{|c|}{\cellcolor[HTML]{a7c957}$1$} & \cellcolor[HTML]{a7c957}$1$ & $2$                             &  $1$                                       & {\color[HTML]{C0C0C0} $-1$}     & {\color[HTML]{C0C0C0} $-1$}     & {\color[HTML]{C0C0C0} $-1$}                     & {\color[HTML]{C0C0C0} $-1$}        \\ \hline
\multicolumn{1}{|c|}{\cellcolor[HTML]{a7c957}$0$} & \cellcolor[HTML]{a7c957}$1$ & $1$                             & {\color[HTML]{C0C0C0} $-1$}                 & {\color[HTML]{C0C0C0} $-1$}       & {\color[HTML]{C0C0C0} $-1$}       &  {\color[HTML]{C0C0C0} $-1$}               & $0$                                \\ \hline
\multicolumn{1}{|c|}{\cellcolor[HTML]{a7c957}$0$} & \cellcolor[HTML]{a7c957}$2$ & {\color[HTML]{C0C0C0} $-1$}                          & {\color[HTML]{C0C0C0} $-1$}                 & {\color[HTML]{C0C0C0} $-1$}       & {\color[HTML]{C0C0C0} $-1$}       &  {\color[HTML]{C0C0C0} $-1$}               & {\color[HTML]{C0C0C0} $-1$}                                \\ \hline
\end{tabular}
\caption{Dimensions of $V_{=(f,a)}\cap W_{= G}$ when $p=3$.}
\end{table}

When $p=5$,

\begin{table}[H]
\centering
\begin{tabular}{cc|c|c|c|c|c|}
\cline{3-7}
                                                &                           & \cellcolor[HTML]{a7c957}$C_2$ & \cellcolor[HTML]{a7c957}$C_2^2$ & \cellcolor[HTML]{a7c957}$D_4$ & \cellcolor[HTML]{a7c957}$D_6$ & \cellcolor[HTML]{a7c957}$\SL_2(\mathbb{F}_5)$ \\ \hline
\multicolumn{1}{|c|}{\cellcolor[HTML]{a7c957}$2$} & \cellcolor[HTML]{a7c957}$0$ & $3$                             & $2$                                       & $1$                             & $1$                             & {\color[HTML]{C0C0C0} $-1$}                     \\ \hline
\multicolumn{1}{|c|}{\cellcolor[HTML]{a7c957}$1$} & \cellcolor[HTML]{a7c957}$1$ & $2$                             & $1$                                       & {\color[HTML]{C0C0C0} $-1$}     & {\color[HTML]{C0C0C0} $-1$}     & {\color[HTML]{C0C0C0} $-1$}                     \\ \hline
\multicolumn{1}{|c|}{\cellcolor[HTML]{a7c957}$0$} & \cellcolor[HTML]{a7c957}$1$ & $1$                             & {\color[HTML]{C0C0C0} $-1$}               & {\color[HTML]{C0C0C0} $-1$}     & {\color[HTML]{C0C0C0} $-1$}     & {\color[HTML]{C0C0C0} $-1$}                     \\ \hline
\multicolumn{1}{|c|}{\cellcolor[HTML]{a7c957}$0$} & \cellcolor[HTML]{a7c957}$2$ & {\color[HTML]{C0C0C0} $-1$}     & {\color[HTML]{C0C0C0} $-1$}               & {\color[HTML]{C0C0C0} $-1$}     & {\color[HTML]{C0C0C0} $-1$}     & $0$                                             \\ \hline
\end{tabular}
\caption{Dimensions of $V_{=(f,a)}\cap W_{= G}$ when $p=5$.}
\end{table}
\newpage

When $p\geq 7$,\\
\begin{table}[H]
\centering
\begin{tabular}{cc|c|c|c|c|}
\cline{3-6}
                                                &                           & \cellcolor[HTML]{a7c957}$C_2$ & \cellcolor[HTML]{a7c957}$C_2^2$ & \cellcolor[HTML]{a7c957}$D_4$ & \cellcolor[HTML]{a7c957}$D_6$ \\ \hline
\multicolumn{1}{|c|}{\cellcolor[HTML]{a7c957}$2$} & \cellcolor[HTML]{a7c957}$0$ & $3$ & $2$ & $1$ & $1$ \\ \hline
\multicolumn{1}{|c|}{\cellcolor[HTML]{a7c957}$1$} & \cellcolor[HTML]{a7c957}$1$ & $2$ & $1$ & {\color[HTML]{C0C0C0} $-1$} & {\color[HTML]{C0C0C0} $-1$} \\ \hline
\multicolumn{1}{|c|}{\cellcolor[HTML]{a7c957}$0$} & \cellcolor[HTML]{a7c957}$1$ & $1$ & {\color[HTML]{C0C0C0} $-1$} & {\color[HTML]{C0C0C0} $-1$} & {\color[HTML]{C0C0C0} $-1$} \\ \hline
\multicolumn{1}{|c|}{\cellcolor[HTML]{a7c957}$0$} & \cellcolor[HTML]{a7c957}$2$ & $0$ iff $p\geq29$ & $0$ iff $p\geq17$ & $0$ iff $p\geq13$ & $0$ iff $p\geq11$ \\ \hline
\end{tabular}
\caption{Dimensions of $V_{=(f,a)}\cap W_{= G}$ when $p\geq7$ (I).}

\end{table}

\begin{table}[H]
\centering
\begin{tabular}{cc|c|c|c|}
\cline{3-5}
                                                &                           & \cellcolor[HTML]{a7c957}$C_3\rtimes D_4$ & \cellcolor[HTML]{a7c957}$\GL_2(\mathbb{F}_3)$ & \cellcolor[HTML]{a7c957}$C_{10}$ \\ \hline
\multicolumn{1}{|c|}{\cellcolor[HTML]{a7c957}$2$} & \cellcolor[HTML]{a7c957}$0$ & $0$ iff $p\equiv 1\pmod 6$ & $0$ iff $p\equiv1,3 \pmod 8$ & $0$ iff $p\equiv1\phantom{,2} \pmod 5$ \\ \hline
\multicolumn{1}{|c|}{\cellcolor[HTML]{a7c957}$1$} & \cellcolor[HTML]{a7c957}$1$ & {\color[HTML]{C0C0C0} $-1$} & {\color[HTML]{C0C0C0} $-1$} & {\color[HTML]{C0C0C0} $-1$} \\ \hline
\multicolumn{1}{|c|}{\cellcolor[HTML]{a7c957}$0$} & \cellcolor[HTML]{a7c957}$1$ & {\color[HTML]{C0C0C0} $-1$} & {\color[HTML]{C0C0C0} $-1$} & $0$ iff $p\equiv2,3 \pmod 5$ \\ \hline
\multicolumn{1}{|c|}{\cellcolor[HTML]{a7c957}$0$} & \cellcolor[HTML]{a7c957}$2$ & $0$ iff $p\equiv 5 \pmod 6$ & $0$ iff $ p\equiv 5,7 \pmod 8$ & $0$ iff $p\equiv 4\phantom{,2} \pmod 5$ \\ \hline
\end{tabular}
\caption{Dimensions of $V_{=(f,a)}\cap W_{= G}$ when $p\geq7$ (II).}
\end{table} 
\end{theorem}
\begin{proof}
The values for $\dim(V_{=(2,0)}\cap W_{= G})$ where $G\in\{C_2,C_2^2,D_4, D_6\}$ follow from the description of the dimensions of the automorphism strata in Section \ref{automorphism_section}.\\

Likewise, the values for $\dim(V_{=(f,a)}\cap W_{= C_2})$ where $(f,a)\in\{(2,0),(1,1),(0,1)\}$ follow from the description of the dimensions of the Ekedahl-Oort strata given in Section \ref{EO_section}.\\

The $p$-ranks of the zero-dimensional strata of Subsection \ref{zerostrata_subsection} were computed by Ibukiyama, Katsura and Oort \cite{Ibukiyama1986SupersingularNumbers}. \\

Finally, the second and third rows of these tables boil down to the following result:
\begin{proposition}\cite[Propositions 1.3 and 1.10]{Ibukiyama1986SupersingularNumbers}
Let $\calC$ be a genus two curve with $\Aut(\calC)> C_2^2$. Then, the $p$-rank of $\Jacc$ is either zero or two. Furthermore, if $\Jacc$ is supersingular and $\Aut(\calC)\geq C_2^2$, then $\Jacc$ is superspecial.
\end{proposition}
The proof of this relies on the fact that, as shown in Subsection \ref{C22_section}, if $\Aut(\calC)\geq C_2^2$, then $\Jacc$ is isogenous to the product of two elliptic curves. Moreover, if $\Aut(\calC)> C_2^2$ one can easily see that these two elliptic curves must be isogenous, and therefore they must have the same $p$-rank, which implies that $f(\Jacc)\neq 1$. As a consequence of this statement, we deduce that $V_{=1,1}\cap W_{=G}=\varnothing$ unless $G=C_2$ or $C_2^2$, and $V_{=0,1}\cap W_{=G}=\varnothing$ unless $G=C_2$.\\

Finally, we know that $\dim(V_{=1,1}\cap W_{=C_2^2})=1$, because $\dim(V_{\leq1})=\dim(W_{\geq C_2^2})=2$ and these two strata intersect transversely, as will become apparent in the proof of Theorem \ref{irreducible_components_theorem}.\qedhere\\
\end{proof}
From this table, we can draw conclusions that may be difficult to prove through other methods. For instance,
\begin{corollary}
There are no Jacobians of genus two curves with $p$-rank one and an automorphism of order greater than two.\\
\end{corollary}

\subsection{Irreducible components of the intersections}

A natural related question that may arise after Theorem \ref{intersection_theorem} is the following:

 \begin{question}
 Are the intersections of these strata irreducible? If not, how many different irreducible components does the Zariski closure have?
 \end{question}\newpage
The answer is given by the following theorem:
\begin{theorem} \label{irreducible_components_theorem}
If $p\leq5$, the Zariski closures of the non-empty strata $V_{=(f,a)}\cap W_{= G}$ are all irreducible.\\

If $p\geq 7$ the number of irreducible components of $\,\overline{V_{=(f,a)}\cap W_{= G}}$ are the following:\\
\begin{table}[H]
\centering
\begin{tabular}{cc|c|c|c|c|}
\cline{3-6}
                                                &                           & \cellcolor[HTML]{a7c957}$C_2$ & \cellcolor[HTML]{a7c957}$C_2^2$ & \cellcolor[HTML]{a7c957}$D_4$ & \cellcolor[HTML]{a7c957}$D_6$ \\ \hline
\multicolumn{1}{|c|}{\cellcolor[HTML]{a7c957}$2$} & \cellcolor[HTML]{a7c957}$0$ & $1$ & $1$ & $1$ & $1$ \\ \hline
\multicolumn{1}{|c|}{\cellcolor[HTML]{a7c957}$1$} & \cellcolor[HTML]{a7c957}$1$ & $1$ & $n_{(1,1),C_2^2}$ & {\color[HTML]{C0C0C0} $0$} & {\color[HTML]{C0C0C0} $0$} \\ \hline
\multicolumn{1}{|c|}{\cellcolor[HTML]{a7c957}$0$} & \cellcolor[HTML]{a7c957}$1$ & $n_{(0,1),C_2}$ & {\color[HTML]{C0C0C0} $0$} & {\color[HTML]{C0C0C0} $0$} & {\color[HTML]{C0C0C0} $0$} \\ \hline
\multicolumn{1}{|c|}{\cellcolor[HTML]{a7c957}$0$} & \cellcolor[HTML]{a7c957}$2$ & $n_{(0,2),C_2}$ & $n_{(0,2),C_2^2}$ & $n_{(0,2),D_4}$ & $n_{(0,2),D_6}$ \\ \hline
\end{tabular}
\caption{Number of irreducible components of $\overline{V_{=(f,a)}\cap W_{= G}}$ when $p\geq7$ (I).}
\label{tableI}
\end{table}

\begin{table}[H]
\centering
\begin{tabular}{cc|c|c|c|}
\cline{3-5}
                                                &                           & \cellcolor[HTML]{a7c957}$C_3\rtimes D_4$ & \cellcolor[HTML]{a7c957}$\GL_2(\mathbb{F}_3)$ & \cellcolor[HTML]{a7c957}$C_{10}$ \\ \hline
\multicolumn{1}{|c|}{\cellcolor[HTML]{a7c957}$2$} & \cellcolor[HTML]{a7c957}$0$ & $1$ iff $p\equiv 1\pmod 6$ & $1$ iff $p\equiv1,3 \pmod 8$ & $1 \text{ iff } p \equiv 1\phantom{,2} \pmod 5$ \\ \hline
\multicolumn{1}{|c|}{\cellcolor[HTML]{a7c957}$1$} & \cellcolor[HTML]{a7c957}$1$ & {\color[HTML]{C0C0C0} $0$} & {\color[HTML]{C0C0C0} $0$} & {\color[HTML]{C0C0C0} $0$} \\ \hline
\multicolumn{1}{|c|}{\cellcolor[HTML]{a7c957}$0$} & \cellcolor[HTML]{a7c957}$1$ & {\color[HTML]{C0C0C0} $0$} & {\color[HTML]{C0C0C0} $0$} & $1$ iff $p\equiv2,3 \pmod 5$ \\ \hline
\multicolumn{1}{|c|}{\cellcolor[HTML]{a7c957}$0$} & \cellcolor[HTML]{a7c957}$2$ & $1$ iff $p\equiv 5 \pmod 6$ & $1$ iff $ p\equiv 5,7 \pmod 8$ & $1 \text{ iff } p \equiv 4\phantom{,2} \pmod 5$ \\ \hline
\end{tabular}
\caption{Number of irreducible components of $\overline{V_{=(f,a)}\cap W_{= G}}$ when $p\geq7$ (II).}
\label{tableII}
\end{table}

The values of $n_{(f,a),G}$ in the table \ref{tableI} are\\
\begin{adjustbox}{width=\textwidth}
\begin{minipage}{1.1\textwidth}
\begin{align*}
n_{(1,1),C_2^2}&=\frac{p-1}{12}+\frac{1-\left(\frac{-1}{p}\right)}{4}+\frac{1-\left(\frac{-3}{p}\right)}{3},\\
n_{(0,1),C_2}&= \frac{p^2 - 1}{2880} 
+ \frac{(p+1)\left( 1 - \left( \frac{-1}{p} \right) \right)}{64}
+ \frac{5(p - 1)\left( 1 + \left( \frac{-1}{p} \right) \right)}{192}
+ \frac{(p + 1)\left( 1 - \left( \frac{-3}{p} \right) \right)}{72}
+ \frac{(p - 1)\left( 1 + \left( \frac{-3}{p} \right) \right)}{36}\\
&+ \left(\frac{2}{5} \text { iff } p \equiv 2, 3\pmod{5}\right)+\left(\frac{1}{4} \text { iff } p \equiv 3, 5\pmod{8}\right)+\left(\frac{1}{6} \text { iff } p \equiv 5\pmod{12}\right),\\
n_{(0,2),C_2}&=\frac{(p-1)\left(p^2-35 p+346\right)}{2880}-\frac{1-\left(\frac{-1}{p}\right)}{32}-\frac{1-\left(\frac{-2}{p}\right)}{8}-\frac{1-\left(\frac{-3}{p}\right)}{9}-\left(\frac{1}{5} \text { iff } p \equiv 4 \pmod 5\right),\\
n_{(0,2),C_2^2}&=\frac{(p-1)(p-17)}{48}+\frac{1-\left(\frac{-1}{p}\right)}{8}+\frac{1-\left(\frac{-2}{p}\right)}{2}+\frac{1-\left(\frac{-3}{p}\right)}{2},\\
n_{(0,2),D_4}&=\frac{p-1}{8}-\frac{1-\left(\frac{-1}{p}\right)}{8}-\frac{1-\left(\frac{-2}{p}\right)}{4}-\frac{1-\left(\frac{-3}{p}\right)}{2},\\
n_{(0,2),D_6}&=\frac{p-1}{6}-\frac{1-\left(\frac{-2}{p}\right)}{2}-\frac{1-\left(\frac{-3}{p}\right)}{3},\\
\end{align*}
\end{minipage}
\end{adjustbox}
where $\left(\frac{n}{p}\right)$ is the Legendre symbol.
\end{theorem}

\begin{proof}
We have seen in Section \ref{automorphism_section} that if $G=C_2, C_{2}^2, D_4$ or $D_6$, $\overline{V_{=(2,0)}\cap W_{= G}}=W_{\geq G}$ are all irreducible.\\

We know that $\overline{V_{=(1,1)}\cap W_{= C_2}}=T_{\geq1}$ is irreducible as a consequence of the work of van der Geer, who proved that $T_{\geq a}\subset\calA_g$ is irreducible whenever $a<g$ \cite[Theorem 2.11]{vanderGeer1999CyclesVarieties}.\\

The fact that the supersingular locus of $\calA_2$ is not irreducible when $p\geq13$ is due to Katsura and Oort, who proved that the number of irreducible components $n_{(0,1),C_2}$ is $H_2(1,p)$, the class number of the non-principal genus \cite{Katsura1987FamiliesSurfaces}. The closed formula for $H_2(1,p)$ has been computed by Hashimoto and Ibukiyama \cite{Hashimoto1980OnForms}.\\

The values of all the zero-dimensional components (the last row of Table \ref{tableI} and all the values of Table \ref{tableII}) are again work of Ibukiyama, Katsura and Oort \cite[Propositions 1.11-1.13 and Theorem 3.3]{Ibukiyama1986SupersingularNumbers}.\\

Our main contribution to this table is computing the value of $n_{(1,1),C_2^2}$.
\end{proof}

\begin{proposition}
The stratum $W_{\geq C_2^2} \cap V_{\leq 1}$ is the union of $n$ rational curves, where $n$ is the number of supersingular elliptic curves in characteristic $p$, which is $n=1$ if $p=2,3$ or $5$ and
\begin{align*}
n&=\frac{p-1}{12}+\frac{1-\left(\frac{-1}{p}\right)}{4}+\frac{1-\left(\frac{-3}{p}\right)}{3}
\end{align*}
otherwise.
\end{proposition}
\begin{proof}
When $p=2$, we can use the results in Section \ref{EO_section} to show that $V_{\leq 1}=\mathbb{V}(J_2,J_4)$, and by computing the intersection with $W_{\geq C_2^2}$, we get 
\begin{align*}
W_{\geq C_2^2} \cap V_{\leq 1}=W_{\geq C_2^5}=V_{\leq0}=\mathbb{V}(J_2,J_4,J_6),
\end{align*} 
which is irreducible.\\

If $p\neq2$, from the description in Subsection \ref{connection_subsection}, we saw that there was a morphism $\pi:\mells\rightarrow W_{\geq C_2^2}$ which was an isomorphism when restricted to the open subset $W_{\geq C_2^2}\setminus W_{D_4}$. We also proved that there was an isomorphism $\Phi$ between $\mquot$ and $\mells$, which is the quotient of 
\begin{align*}
\mell=((X(2)\times X(2))\setminus\Delta)/S_3
\end{align*}
by the action of $C_2$ that swaps around the copies of $X(2)$. Furthermore, we saw in Section \ref{C22_section} that if $\Phi((\calC,\sigma))=\{(E_1,\phi_1),(E_2,\phi_2)\}$, $\Jacc$ is isogenous to $E_1\times E_2$. Therefore, $\calC\in W_{\geq C_2^2} \cap V_{\leq 1}$ if and only if $f(E_1\times E_2)\leq1$.  As $f(E_1\times E_2)=f(E_1)+f(E_2)$, we deduce that this happens if and only if at least one of $E_1$ or $E_2$ is supersingular. 

Let $E_0$ be a supersingular elliptic curve, and $\phi_0$ a level 2 structure on $E_0$.
Consider the following subvariety\\
\begin{adjustbox}{width=\textwidth}
\begin{minipage}{1.1\textwidth}
\begin{align*}
E_0\times X(2)=\{\{(E_0,\phi_0),(E,\phi)\} \text{ for some }(E,\phi) \in X(2) \text{ with } (E,\phi) \neq(E_0,\phi_0)\}\subset\mells.\\
\end{align*}
\end{minipage}
\end{adjustbox}

Note that the subscheme $E_0\times X(2)$ does not depend on the choice of the level 2 structure $\phi_0$. This is because, as we have seen, if $\phi_0$ and $\phi_0'$ are two structures on $E_0$, there exists a $\tau\in \Aut((\mathbb{Z}/2\mathbb{Z})^2)$ such that $\phi_0'=\phi_0\circ\tau$. But then, we have that for any $(E,\phi)\in X(2)$,
\begin{align*}
\{(E_0,\phi_0'),(E,\phi)\}=\{(E_0,\phi_0),(E,\phi\circ\tau^{-1})\}
\end{align*}
in $\mells$. We know that $E_0\times X(2)$ is irreducible, as it is the image under the quotient map of a variety isomorphic to $X(2)$, which is irreducible.\\

Let $C_{E_0} :=\Phi^{-1}(E_0\times X(2))\subset \mquot$. From the fact that $\Phi$ is an isomorphism and that any pair of elliptic curves where at least one of them is supersingular lies inside a copy of $E_0\times X(2)$ for some $E_0$, we deduce that any pair $(\calC,\sigma)$ with $\calC\in W_{\geq C_2^2} \cap V_{\leq 1}$ must lie inside a $C_{E_0}$ for some supersingular $E_0$. From the fact that the morphism $\pi:\mells\rightarrow W_{\geq C_2^2}$ is finite and surjective, we deduce that $\pi(C_{E_0})$ is also an irreducible curve, and that
\begin{align*}
W_{\geq C_2^2} \cap V_{\leq 1}=\bigcup_{E_0\text{ supersingular}} \pi(C_{E_0}).\\
\end{align*}

Hence, the number of irreducible components of $W_{\geq C_2^2} \cap V_{\leq 1}$ is the same as the number of supersingular elliptic curves in characteristic $p$. This quantity was computed by Igusa \cite{Igusa1958ClassDiscriminant}.
\end{proof}
\vspace{10pt}
\section{Future directions}
It would be interesting to compute the intersections of the Ekedahl-Oort and the automorphism strata inside either the moduli space of genus $3$ curves $M_3$, or the hyperelliptic locus $\calH_3$.\\\newpage
\bibliographystyle{alpha}
\bibliography{references.bib}

\end{document}